\documentclass[12pt]{amsart}
\usepackage{geometry}
\usepackage{euscript,eufrak,verbatim, mathrsfs}
\usepackage[psamsfonts]{amssymb}
\usepackage{bbm}
\usepackage{graphicx}
\usepackage{float, tikz}
\usepackage{pgfplots}
\pgfplotsset{compat=1.18}
\usepgfplotslibrary{fillbetween}
\usepackage{caption}
\usepackage{subcaption}
\usepackage{extarrows}
\usepackage[all, cmtip]{xy}

\usepackage{upref, xcolor, dsfont}
\usepackage{amsfonts,amsmath,amstext,amsbsy, amsopn,amsthm}
\usepackage{enumerate}

\usepackage{url}
\usepackage{mathtools}

\usepackage{helvet}
\usepackage{courier}
\usepackage{type1cm}

\usepackage{multicol}
\usepackage[bottom]{footmisc}

\usepackage[colorlinks,citecolor = red, linkcolor=blue,hyperindex]{hyperref}
\usepackage{bookmark}

\newtheorem{theorem}{Theorem}[section]
\newtheorem*{theorem*}{Theorem A}
\newtheorem{lemma}[theorem]{Lemma}

\newtheorem{proposition}[theorem]{Proposition}
\newtheorem{corollary}[theorem]{Corollary}
\newtheorem*{definition*}{Definition}
\newtheorem*{remark*}{Remark}

\newtheorem*{observation*}{Observation}

\newtheorem*{assumption*}{Assumption}

\theoremstyle{definition}
\newtheorem{definition}{Definition}[section]

\newtheorem*{problem*}{Problem}

\theoremstyle{remark}
\newtheorem{remark}{Remark}[section]
\newtheorem{claim}{Claim}[section]

\newcommand{\D}{\mathbb{D}}
\newcommand{\C}{\mathbb{C}}

\newcommand{\Hol}{\mathrm{Hol}}

\newcommand{\norm}[1]{\left\lVert #1\right\rVert}

\newcommand{\dd}{\mathrm{d}}

\begin{document}

\title[Strong and weak estimates]{Strong and weak-type estimates for radial weighted Bergman projections}

\author{Yuerang Li}
\address{Yuerang Li: College of Mathematics and Statistics, Chongqing University, Chongqing 401331, China}
\email{yuerangli@outlook.com}

\author{Zipeng Wang}
\address{Zipeng Wang: College of Mathematics and Statistics, Chongqing University, Chongqing 401331, China}
\email{zipengwang2012@gmail.com; zipengwang@cqu.edu.cn}

\subjclass{Primary 47B38, 46E22, 30H20.}
\keywords{Weighted Bergman projections, $\widehat{\mathcal{D}}$-weights, $L^p$-bounded, weak-type (1,1).}

\begin{abstract}
We completely characterize the $L^p$-boundedness and the weak-type (1,1) estimate of radial weighted Bergman projections on the unit disk. Our result, in particular, confirms a conjecture proposed by Pel\'{a}ez and R\"{a}tty\"{a} in 2021 and thereby settles a longstanding problem in the area that was formally posed by Dostani\'{c} in 2004.
Consequently, we establish the dichotomy that a radial weighted Bergman projection is bounded either only for $p=2$, or for all $p\in(1,\infty)$.

\end{abstract}

\maketitle

\section{Introduction and main results}
Bergman projections are fundamental operators in the theory of analytic function spaces and their operators. Their boundedness and endpoint behavior are central problems in the theory. The boundedness of Bergman projections plays an important role in identifying dual Bergman spaces under the natural integral pairing. Via reproducing kernels, Bergman projections can be viewed as singular integral operators, placing them at the interface of complex analysis and harmonic analysis \cite{BB78,McN94,Hed02,Bor04,PR14,PR15,PR16,Pel16}.

A major longstanding problem in the area, highlighted in \cite{PR21}, seeks to characterize the radial weights $\omega$ on the unit disk for which the associated Bergman projection $P_\omega$ is bounded on $L_\omega^p$ for $1<p<\infty$. Although it had been informally known to experts for some time, the problem was formally posed by Dostanić in 2004 \cite{Dos04}. The two natural limiting cases concern the boundedness of $P_\omega:L^1_\omega\to L^{1,\infty}_\omega$ and $P_\omega:L^\infty\to \mathcal{B}$, where $L_\omega^{1,\infty}$ denotes the weighted weak-$L^1$ space and $\mathcal{B}$ is the Bloch space. 

Peláez and Rättyä in \cite{PR21} proved that $P_\omega:L^\infty\to \mathcal{B}$ is bounded if and only if $\omega$ is a $\widehat{\mathcal{D}}$-weight, and further proved that the $\widehat{\mathcal{D}}$-weighted Bergman projection is bounded on $L^p_\omega$ for every $1<p<\infty$. Moreover, they conjectured \cite[Conjecture 2]{PR21} that $P_\omega:L^p_\omega\to L^p_\omega$ is bounded for some $p\in(1,\infty)\setminus\{2\}$ precisely when $\omega\in\widehat{\mathcal D}$. The conjecture had previously been verified in several cases under additional regularity assumptions \cite{PR21}.

The purpose of this paper is to resolve the Peláez--Rättyä conjecture and characterize the weak-$(1,1)$ estimate, which again is the $\widehat{\mathcal{D}}$-condition. Combined with the $L^\infty$-to-Bloch characterization of Peláez and Rättyä, our results complete the strong and weak-type estimates for radial weighted Bergman projections.

A weight $\omega$ is a nonnegative integrable function on $[0,1)$. It induces a radial measure on $\D$ by setting $\omega(z)=\omega(|z|)$ for any $z\in\D.$ For $1\le p<\infty$, the weighted Lebesgue space $L_\omega^p$ consists of all measurable functions $f$ on $\D$ such that
\[
\|f\|_{L_\omega^p}^{p} = \int_{\D} |f(z)|^{p}\omega(z)\,\dd A(z) < \infty,
\]
where $\dd A(z)= \dd x\,\dd y/\pi$ denotes the normalized area measure on $\D$. Moreover, the weighted weak space $L^{1,\infty}_\omega$ is the set of all measurable functions $f$ on the unit disk such that
$$
    \|f\|_{L^{1,\infty}_\omega}
    =
    \sup_{t>0}
    t\,
    \omega\bigl(\{z\in\D : |f(z)|>t\}\bigr)<\infty.
$$

Let $\Hol(\D)$ be the space of analytic functions on the unit disk. The weighted Bergman space $A_\omega^p$ is defined as
$
A_\omega^p=\Hol(\D)\cap L_\omega^p.
$
The necessary and sufficient condition for $A_\omega^2$ to be a reproducing kernel Hilbert space is 
that the weight $\omega$ satisfies
\[
\widehat{\omega}(z)=\int_{|z|}^1 \omega(s)\,\dd s > 0,
\quad z\in\D,
\]
which is assumed throughout the paper. Beyond their intrinsic interest, radial weighted Bergman spaces arise naturally from Shimorin's representation of the Bergman kernel for log-subharmonic weights \cite{Sh02,HJS02,HKZ00,PRW19}, from Guo's theory of invariant unitary reproducing kernels on Hilbert modules \cite{GHX04} and from other topics on analytic function spaces \cite{PR14}.

For a radial weight $\omega$, the Bergman projection $P_\omega: L_\omega^2 \to A_\omega^2$ is given by
\begin{equation*}
(P_\omega f)(z)=\int_{\D} f(\zeta)\overline{B_z^\omega(\zeta)}\,\omega(\zeta)\,\dd A(\zeta),
\end{equation*}
where $B_z^\omega$ is the reproducing kernel of $A_\omega^2$ and admits a series representation
\[
B_z^\omega(\zeta) = \sum_{n=0}^{\infty} \frac{(\overline{z}\zeta)^n}{2\omega_{2n+1}}, \quad z,\zeta\in\mathbb D.
\]
Here and throughout,
\[
\omega_x = \int_0^1 r^x\omega(r)\,\dd r, \quad x\geq 0,
\]
denotes the $x$-th moment of $\omega$.

For \(\alpha>-1\), the weight
$
W_\alpha(z)=(\alpha+1)(1-|z|^2)^\alpha
$
is referred to as a standard weight, while $\alpha=0$ corresponds to the classical Bergman space $A^2$. In this case, the Bergman projection $P_B$ is given by
\[
P_Bf(z)
  =
  \int_{\mathbb{D}}
  \frac{f(\zeta)}
  {(1-z\overline{\zeta})^2}\,dA(\zeta).
\]
 
Zaharjuta and Judovic \cite{ZJ64} proved that \(P_B\) is bounded on \(L^p\) for \(1<p<\infty\). The current canonical proof, due to Forelli and Rudin \cite{FR74}, is known as the Forelli--Rudin estimate, which has become a standard technique in the theory of operators on function spaces \cite{Zhu07}. 

Note that \(P_B\) is not bounded on \(L^\infty\), and hence, by duality, not on \(L^1\). For \(p=\infty\), Coifman--Rochberg--Weiss \cite{CRW76} showed that \(P_B\) maps \(L^\infty\) boundedly onto the Bloch space \(\mathcal{B}\).  Recall that for an analytic \(f\in\operatorname{Hol}(\mathbb{D})\), its Bloch seminorm is 
$
\|f\|_\mathcal{B}=\sup_{z\in\mathbb{D}}(1-|z|^2)|f'(z)|.
$
By a result of Coifman--Rochberg--Weiss in \cite[p. 632]{CRW76}, the Bloch space is the natural analytic BMO-type endpoint space.
At the other endpoint, \(p=1\), the weak-type \((1,1)\) estimate for \(P_B\) was proved by Deng, Huang, Zhao, and Zheng \cite{DHZZ01}. The Bergman projection for standard weights has the aforementioned properties as the unweighted one \cite{Zhu07}.

The situation for general radial weights may be quite different. Motivated by the work of Lin and Rochberg \cite{LR96} on Toeplitz and Hankel operators, Dostani\'c \cite{Dos04} proved that the Bergman projection corresponding to the weight
\[
w(r) = (1-r^2)^A \exp\left(-\frac{B}{(1-r^2)^\alpha}\right),\quad A\ge 0,\ B>0,\ 0<\alpha\le 1,
\]
is bounded only for \(p=2\). He further studied this problem in \cite{Dos09}. Zeytuncu \cite{Ze13} extended Dostani\'c's example to a broad class of radial weights. Further examples of the Bergman projection that is only bounded for $p=2$ can be found in \cite{CP15} and \cite{BLT21}.

Recall that a radial weight \(\omega\) belongs to the class
\(\widehat{\mathcal D}\) if its tail integral \(\widehat{\omega}\) satisfies the one-sided doubling condition
\begin{equation}\label{eq:widehat-D-definition}
  \widehat{\omega}(r) \le C\widehat{\omega}\left(\frac{1+r}{2}\right), \quad 0\le r<1,
\end{equation}
for some constant \(C=C(\omega)\ge1\). 

In general, Bergman projections are not bounded on either \(L^1\) or \(L^\infty\) \cite{Da22}. For \(p=\infty\), recall that Pel\'aez and R\"atty\"a \cite{PR21} showed that a radial weighted Bergman projection \(P_\omega: L^\infty \to \mathcal{B}\) is bounded if and only if \(\omega \in \widehat{\mathcal{D}}\). For $p=1$, Pel\'aez, R\"atty\"a and Wick \cite{PRW19} studied the weak-type (1,1) estimate for a class of $\widehat{\mathcal{D}}$-weights whose Bergman kernels admit an integrable representation. The general weak-type (1,1) problem had remained open.

In this work, we completely resolve all remaining cases of strong and weak-type estimates for radial weighted Bergman projections. 
\begin{theorem}\label{thm:main-ZS}
Let \(\omega\) be a radial weight on \(\mathbb D\). The following statements are equivalent:
\begin{enumerate}
\item \(P_\omega:L^1_\omega\to L^{1,\infty}_\omega\) is bounded;
\item \(P_\omega:L_\omega^p\to L_{\omega}^p\) is bounded for every \(1<p<\infty\);
\item \(P_\omega:L_\omega^p\to L_{\omega}^p\) is bounded for some \(p\in(1,\infty)\setminus\{2\}\);
%\item \(P_\omega:L^\infty\to \mathcal{B}\) is bounded;
\item \(\omega\in\widehat{\mathcal D}\).
\end{enumerate}
\end{theorem}

As an immediate corollary, we obtain the following dichotomy.

\begin{corollary}
Let \(\omega\) be a radial weight on \(\mathbb D\). The Bergman projection $P_\omega$ is either bounded only for $p=2$ or bounded for all $p\in(1,\infty)$.
\end{corollary}

\begin{remark}
Pel\'aez and R\"atty\"a \cite[Theorem 7]{PR21} proved $(4)\Rightarrow(2)$. Our argument also yields a new proof of this
implication.
\end{remark}

The Bergman projection $P_\omega$ is always bounded on $L^2_\omega$. By the Marcinkiewicz interpolation theorem and duality, this yields $(1)\Rightarrow(2)$, while $(2)\Rightarrow(3)$ is immediate. The implication $(3)\Rightarrow(4)$ proves the Pel\'aez--R\"atty\"a  conjecture, thereby resolving Dostani\'c's problem. The argument hinges on a rigidity principle for moment functions: a reverse H\"{o}lder-type inequality at three distinct scales enforces a doubling condition on the moment sequence, which in turn implies the $\widehat{\mathcal{D}}$-condition. Finally, to complete the proof of $(4)\Rightarrow(1)$, we embed the weighted Bergman projection into the framework of non-homogeneous Calder\'{o}n--Zygmund operators. 

Although the measure induced by a $\widehat{\mathcal D}$-weight need not be doubling on the whole disk, it admits an upper-doubling
property near the boundary. Moreover, the singular behavior of the Bergman kernel is also concentrated near the boundary diagonal. This makes the non-homogeneous Calder\'on--Zygmund theory on upper-doubling metric measure spaces particularly suitable. Using the framework developed by Hyt\"onen, Liu, Yang, and Yang \cite{HLYY12}, we establish the desired weak-type estimate.

At the technical level, we need to show that the integral kernel of the $\widehat{\mathcal{D}}$-weighted Bergman projection, which is the reproducing kernel of the $\widehat{\mathcal{D}}$-weighted Bergman space, satisfy certain modified size and smoothness estimates. In recent work \cite{PRWW26}, Pennanen, Rättyä, Wang, and Wu provide an elegant and sharp estimate of all orders of the $\widehat{\mathcal{D}}$-reproducing kernel. We include a different proof of the estimate needed for our current applications in Section \ref{S:S4} to make the paper self-contained. 

Another integral operator closely related to $P_\omega$ is the maximal weighted Bergman projection
\begin{equation*}
(P_\omega^+ f)(z)=\int_{\mathbb D} f(\zeta)|B_z^\omega(\zeta)|\,\omega(\zeta)\,dA(\zeta),
\end{equation*}
where $B_z^\omega$ is the reproducing kernel of $A_\omega^2$. For the standard weight $W_\alpha$, the operators $P_{W_\alpha}$ and $P^+_{W_\alpha}$ share the same strong-type and weak-type estimates. However, Peláez and Rättyä \cite[Corollary 10]{PR21} showed that there exists a radial weight for which $P_\omega$ is bounded on $L^p_\omega$ but $P_\omega^+$ is not. In fact, \cite[Theorem 9]{PR21} yields a stronger characterization: for a $\widehat{\mathcal D}$-weight and $1<p<\infty$, $P_\omega^+:L^p_\omega\to L^p_\omega$ is bounded if and only if $\omega\in\mathcal D$ where $\mathcal D$ denotes the subclass of $\widehat{\mathcal D}$ consisting of those weights for which there exist constants $C_1>0$ and $C_2>1$ such that
\[
\widehat{\omega}(r)\geq C_1\widehat{\omega}(1-\frac{1-r}{C_2})\qquad \forall 0<r<1.
\]

We obtain the following complete characterization of the strong and weak-type boundedness of the maximal weighted Bergman projection.
\begin{corollary}\label{thm:P+main-ZS}
Let \(\omega\) be a radial weight on \(\mathbb D\). The following statements are equivalent:
\begin{enumerate}
\item \(P_\omega^+:L^1_\omega\to L^{1,\infty}_\omega\) is bounded;
\item \(P_\omega^+:L_\omega^p\to L_{\omega}^p\) is bounded for every \(1<p<\infty\);
\item \(P_\omega^+:L_\omega^p\to L_{\omega}^p\) is bounded for some \(p\in(1,\infty)\setminus\{2\}\);
%\item \(P_\omega:L^\infty\to \mathcal{B}\) is bounded;
\item \(\omega\in \mathcal D\).
\end{enumerate}
\end{corollary}

%\begin{remark}
%We recall the elegant result of Pel\'{a}ez and R\"{a}tty\"{a} \cite{PR21} that establishes the equivalence $(4)\Leftrightarrow(5)$ and the implication $(5)\Rightarrow (2)$. Observe that $(2)\Rightarrow (3)$ is trivial. The contributions of the present paper are the necessity of the $\widehat{\mathcal{D}}$-condition for $L_\omega^p$-boundedness when $p\not=2$, and the weak-(1,1) characterization.
%\end{remark}
%\begin{remark}
%The BMO seminorm of the function $f$ defined on $\mathbb{D}$ is given by
%\[
%\|f\|_{\text{BMO}}=\sup_{Q\subset\mathbb{D}}\frac{1}{|Q|}\int_Q |f(z)-\frac{1}{|Q|}\int_Q f(\zeta)dA(\zeta)|dA(\zeta),
%\]
%where the supremum runs over all cubes $Q$ that are contained in $\mathbb{D}$.
%By a theorem of Coifman–Rochberg–Weiss \cite[p. 632]{CRW76}, for an analytic function $\|f\|_\mathcal{B}\approx \|f\|_{\text{}BMO}$.
%\end{remark}

The remainder of the paper is organized as follows. In Section~\ref{S:S2}, we prove the implication $(3)\Rightarrow (4)$ in Theorem~\ref{thm:main-ZS} by deriving the $\widehat{\mathcal D}$-condition from a reverse H\"{o}lder-type inequality for moments of the weight. Section~\ref{S:S3} completes the implication $(4)\Rightarrow (1)$ in Theorem~\ref{thm:main-ZS} using the non-homogeneous Calder\'{o}n--Zygmund theory. The proof of Corollary~\ref{thm:P+main-ZS} is given in Section~\ref{S:corollaryP+}. Section~\ref{S:imply14} provides a direct proof of $(1)\Rightarrow (4)$ in Theorem~\ref{thm:main-ZS} using testing functions. Finally, Section~\ref{S:S4} presents the pointwise estimates needed for the $\widehat{\mathcal D}$-weighted Bergman kernel.

\medskip

\noindent \textbf{Acknowledgements.} This work is supported by the National Natural Science Foundation of China (No.12471116) and 2025CDJ-IAIS YB-004 (Chongqing University).

\section{From Strong-type boundedness to the $\widehat{\mathcal D}$-Condition}\label{S:S2}
The goal of this section is to prove the implication $(3)\Rightarrow (4)$ in Theorem \ref{thm:main-ZS}. Namely,
\[
L_\omega^p\text{-boundedness of radial weighted Bergman projection}\Rightarrow \omega\in\widehat{\mathcal{D}}.
\]
The proof consists of three steps. First, $L_\omega^p$-boundedness implies a reverse Hölder inequality for the moments of $\omega$. Second, we prove a rigidity principle showing that such an inequality forces the moment function to be doubling. Finally, moment doubling implies the one-sided doubling condition for the tail integral.

\begin{lemma}\label{lem:Peano-kernel-representation}
Let \(2<p<\infty\), and let $p'=\frac{p}{p-1}$ be the conjugate exponent of \(p\), i.e.,
$
\frac{1}{p}+\frac{1}{p'}=1.
$
Let \(n>0\), and let $H$ be a $C^2$ function on the interval $[p'n,pn]$, then
\[
H(2n)-\frac{1}{p}H(pn)-\frac{1}{p'}H(p'n)
= \int_{p'n}^{pn} K_n(s)\bigl(-H''(s)\bigr)\,\dd s,
\]
where
\[
K_n(s)=
\begin{cases}
\dfrac{s-p'n}{p'}, & p'n\le s\le 2n,\\[6pt]
\dfrac{pn-s}{p}, & 2n\le s\le pn.
\end{cases}
\]
\end{lemma}

\begin{proof}
The proof follows from a direct computation.
Using the definition of \(K_n\) and integrating by parts, we obtain
\[
\begin{aligned}
\int_{p'n}^{pn} K_n(s)\bigl(-H''(s)\bigr)\,\dd s
&=\int_{p'n}^{2n}\frac{s-p'n}{p'}\bigl(-H''(s)\bigr)\,\dd s \\
&\quad+\int_{2n}^{pn}\frac{pn-s}{p}\bigl(-H''(s)\bigr)\,\dd s \\
&=-\left[\frac{s-p'n}{p'}H'(s)\right]_{p'n}^{2n}+\frac{1}{p'}\int_{p'n}^{2n}H'(s)\,\dd s \\
&\quad-\left[\frac{pn-s}{p}H'(s)\right]_{2n}^{pn} -\frac{1}{p}\int_{2n}^{pn}H'(s)\,\dd s \\
&=-\frac{2n-p'n}{p'}H'(2n)+\frac{1}{p'}\bigl(H(2n)-H(p'n)\bigr) \\
&\quad +\frac{pn-2n}{p}H'(2n) -\frac{1}{p}\bigl(H(pn)-H(2n)\bigr).
\end{aligned}
\]
Since \(p'=p/(p-1)\), we have
\[
\frac{2n-p'n}{p'}
= n\frac{2-p'}{p'}
= n\frac{p-2}{p}
= \frac{pn-2n}{p}.
\]
Therefore,
\[
\begin{aligned}
\int_{p'n}^{pn} K_n(s)\bigl(-H''(s)\bigr)\,\dd s
&=
\frac{1}{p'}\bigl(H(2n)-H(p'n)\bigr)
-\frac{1}{p}\bigl(H(pn)-H(2n)\bigr) \\
&=
\left(\frac{1}{p}+\frac{1}{p'}\right)H(2n)
-\frac{1}{p}H(pn)-\frac{1}{p'}H(p'n) \\
&=
H(2n)-\frac{1}{p}H(pn)-\frac{1}{p'}H(p'n).
\end{aligned}
\]
This completes the proof.
\end{proof}

Let $\nu$ be a finite positive Borel measure on $[0,1)$ such that $\nu([r,1))>0$ for $0\le r<1$,
and define
\[
 M_\nu(x)=\int_0^1r^x \, \dd \nu(r), \quad x\ge0.
\]

The next two results show that the one-sided doubling property of a measure follows from a reverse H\"older condition on its moment sequence.
 
\begin{lemma}\label{lem:reverse-Jensen}
Let $\nu$ be a finite positive Borel measure on $[0,1)$ with moment function $M_{\nu}$. Let \(2<p<\infty\), and let $p'$
be the conjugate exponent of \(p\). If there exists a constant $C\ge1$ such that
\begin{equation}\label{eq:abstract-reverse-Holder}
 M_\nu(pn)^{1/p}M_\nu(p'n)^{1/p'} \le C M_\nu(2n), \quad n\in\mathbb N,
\end{equation}
then there exist constants $n_p\in\mathbb N$ and $\kappa_p>0$, depending only on $p$, such that
\begin{equation}\label{eq:real-moment-doubling}
 M_\nu(x) \le C^{2\kappa_p}M_\nu(2x), \quad x\ge n_p+1.
\end{equation}
\end{lemma}

\begin{proof}
By the homogeneity of \eqref{eq:abstract-reverse-Holder}, we may assume without loss of generality that $\nu$ is a probability measure on $[0,1)$. We then define 
\[
H(x) = -\log M_\nu(x), \quad x\ge0.
\]
\begin{claim}\label{claim:properties-of-H}
The function $H$ is nonnegative, nondecreasing, twice continuously differentiable, and concave on $(0,\infty)$. Moreover, it satisfies the asymptotic derivative condition
\begin{equation}\label{eq:boundary-H-prime}
 \lim_{x\to\infty}H'(x)=0.
\end{equation}
\end{claim}
\begin{proof}[Proof of Claim \ref{claim:properties-of-H}]
  Since $M_\nu$ is positive, bounded above by $1$, and nonincreasing, it follows that $H$ is nonnegative and nondecreasing.
  
  Fix a compact interval $[a,b]\subset(0,\infty)$. For $x\in[a,b]$ and $j\in\{1,2\}$,
  \[
  r^x|\log r|^j\le r^a|\log r|^j, \quad 0<r<1.
  \]
  The function on the right, extended by zero at $r=0$, is bounded on $[0,1)$ and hence integrable with respect to the finite measure $\nu$. The dominated convergence theorem therefore shows that $M_\nu$ is twice continuously differentiable on $(0,\infty)$, with
  \[
  M_\nu'(x) = \int_0^1 r^x \log r \, \dd \nu(r), \quad \text{and} \quad M_\nu''(x) = \int_0^1 r^x (\log r)^2 \, \dd \nu(r),
  \]
  where the integrands are understood to be zero at $r=0$.
  Consequently,
  \[
  H''(x)= - (\log M_\nu)''(x) = - \frac{M_\nu''(x) M_\nu(x) - (M_\nu'(x))^2}{M_\nu(x)^2}.
  \]
  To determine the sign of the numerator, we apply the Cauchy--Schwarz inequality:
  \begin{align*}
    (M_\nu'(x))^2 
    &= \left(\int_0^1 r^{x/2} (r^{x/2} \log r) \, \dd \nu(r) \right)^2 \\
    &\le \left(\int_{0}^1 r^x \, \dd \nu(r) \right) \left(\int_0^1 r^x (\log r)^2\, \dd \nu(r)\right) = M_\nu(x) M_\nu''(x).
  \end{align*}
  Thus, $H''(x) \le 0$ for $x>0$, which proves the concavity of $H$.

  Finally, we prove the asymptotic derivative condition \eqref{eq:boundary-H-prime}. Fix $0<\rho<1$. Since $\nu([\rho,1))>0$, we have
  \[
  M_\nu(x) = \int_0^1 r^x \dd \nu(r)  \ge \int_\rho^1 r^x \dd \nu(r) \ge \rho^x \nu([\rho,1)).
  \]
  This implies that
  \[
  H(x) = -\log M_\nu(x) \le - x \log \rho - \log \nu([\rho, 1)).
  \]
  Hence,
  \[
  \frac{H(x)}{x} \le -\log \rho - \frac{\log \nu([\rho,1))}{x}.
  \]
  Taking the limit inferior and limit superior, and using the nonnegativity of $H$, we obtain
  \[
    0 \le \liminf_{x \to \infty} \frac{H(x)}{x} \le \limsup_{x \to \infty} \frac{H(x)}{x} \le - \log \rho, \quad 0 < \rho < 1.
  \]
  Letting $\rho$ tend to $1^{-}$, we obtain $\lim_{x\to \infty} H(x)/x = 0$. Thus, the concavity and monotonicity of $H$ imply that
  \[
  0 \le H'(x) \le \frac{H(x) - H(a)}{x - a}, \quad \text{for } 0 < a < x.
  \]
  This implies that $$\lim_{x\to \infty} H'(x) = 0.$$ This proves the claim.
\end{proof}

Taking logarithms in \eqref{eq:abstract-reverse-Holder}, we obtain
\begin{equation}\label{eq:J-bound}
 J_n :=H(2n)-\frac1pH(pn)-\frac1{p'}H(p'n) \le \log C, \quad n\in\mathbb N.
\end{equation}
By Lemma \ref{lem:Peano-kernel-representation}, we have 
\begin{equation}\label{eq:Peano-kernel}
 J_n =\int_{p'n}^{pn}K_n(s)\bigl(-H''(s)\bigr) \, \dd s, 
\end{equation}
where
\begin{equation}\label{eq:kernel-Kn}
 K_n(s) =\begin{cases} 
  \displaystyle \frac{s-p'n}{p'},&p'n\le s\le2n,\\[6pt]
 \displaystyle \frac{pn-s}{p},&2n\le s\le pn,\\[6pt]
 0,&\text{otherwise}.
 \end{cases}
\end{equation}
Fix $s>0$. The condition $K_k(s)>0$ is equivalent to
\[
 \frac{s}{p}<k<\frac{s}{p'}.
\]
The midpoint of this interval is $s/2$. We retain its central half, which
leaves the same margin from both endpoints. Set
\begin{equation*}
 \delta=\frac{p-2}{4p},
 \quad
 A=\frac12-\delta=\frac{p+2}{4p},
 \quad
 B=\frac12+\delta=\frac{3p-2}{4p}.
\end{equation*}
Then the interval
\[
 \left[As,Bs\right]\subset \left[\frac{s}{p},\frac{s}{p'}\right].
\]
Let $k\in\mathbb N\cap[As,Bs]$. If $k\le s/2$, then
$2k\le s\le pk$, and the second branch of \eqref{eq:kernel-Kn} gives
\[
 K_k(s)=k-\frac{s}{p} \ge As-\frac{s}{p} =\delta s.
\]
If $k\ge s/2$, then $p'k\le s\le2k$, and the first branch gives
\[
 K_k(s)=\frac{s}{p'}-k \ge\frac{s}{p'}-Bs =\delta s.
\]
Hence
\begin{equation}\label{eq:kernel-lower-window}
 K_k(s)\ge\delta s,
 \quad k\in\mathbb N\cap[As,Bs].
\end{equation}

Suppose now that $\delta s\ge1$. The interval $[As,Bs]$ has length
$2\delta s$, so it contains at least $\lfloor2\delta s\rfloor$ integers;
since $\delta s\ge1$, this number is at least $\delta s$. For each of these
integers, $k\le Bs$; therefore \eqref{eq:kernel-lower-window} implies
\begin{align}
 \sum_{\substack{k\in\mathbb N\\As\le k\le Bs}}
 \frac{K_k(s)}{k^2}\ge (\delta s)\frac{\delta s}{B^2s^2} =\frac{\delta^2}{B^2}=\left(\frac{p-2}{3p-2}\right)^2 =: C_1.
 \label{eq:discrete-kernel-lower}
\end{align}
Here the constant $C_1$ depends only on $p$.
We now choose the lower threshold on $n$. Two requirements have appeared:
\begin{itemize}
  \item we need $\delta n\ge1$ in order to use \eqref{eq:discrete-kernel-lower}, and
  \item we shall need $An\ge2$ to estimate a convergent tail.
\end{itemize}
Thus set
\begin{equation*}
 n_p
 =\left\lceil \max\left\{\frac1\delta,\frac2A\right\}\right\rceil
 =\left\lceil\max\left\{\frac{4p}{p-2},\frac{8p}{p+2}\right\}\right\rceil.
\end{equation*}
Fix an integer $n\ge n_p$ and put $m=\lceil An\rceil$. If $s\ge n$ and
$k\in\mathbb N\cap[As,Bs]$, then $k\ge As\ge An$, hence $k\ge m$.
Consequently, \eqref{eq:discrete-kernel-lower} yields
\[
 \sum_{k=m}^{\infty}\frac{K_k(s)}{k^2}\ge C_1 ,
 \quad s\ge n.
\]
Multiplying by the nonnegative function $-H''(s)$, integrating, and then
using Tonelli's theorem together with \eqref{eq:J-bound} and \eqref{eq:Peano-kernel}, we obtain
\begin{align*}
  C_1 \int_n^\infty\bigl(-H''(s)\bigr)\, \dd s
 &\le \int_{n}^{\infty} \sum_{k = m}^{\infty} \frac{K_k(s)}{k^2} (-H''(s))\, \dd s\\
 &\le\sum_{k=m}^\infty\frac1{k^2} \int_0^\infty K_k(s)\bigl(-H''(s)\bigr)\, \dd s \\
 &=\sum_{k=m}^\infty\frac{J_k}{k^2} \le (\log C)\sum_{k=m}^\infty\frac1{k^2}.
\end{align*}
Because $m\ge2$,
\[
 \sum_{k=m}^\infty\frac1{k^2}
 \le\int_{m-1}^\infty\frac{ \dd t}{t^2}
 =\frac1{m-1}.
\]
Moreover, $m-1\ge An/2$ because $An\ge2$. Combining these estimates with
\eqref{eq:boundary-H-prime} gives
\begin{align*}
 H'(n)
 &=\int_n^\infty\bigl(-H''(s)\bigr) \, \dd s \le\frac{\log C}{(m - 1) C_1 } \le \frac{2\log C}{A C_1  n}.
\end{align*}
Therefore, since $H'$ is nonincreasing,
\[
 H(2n)-H(n)
 =\int_n^{2n}H'(t)\, \dd t
 \le nH'(n)
 \le\frac{2\log C}{A C_1 }.
\]
The exponent in the final estimate has now emerged from the proof. Define
\begin{equation*}
 \kappa_p :=\frac{2}{AC_1} =\frac{8p(3p-2)^2}{(p+2)(p-2)^2}.
\end{equation*}
Then
\[
 \log\frac{M_\nu(n)}{M_\nu(2n)}
 =H(2n)-H(n)
 \le\kappa_p\log C,
\]
which proves
\begin{equation}\label{eq:integer-moment-doubling}
 M_\nu(n) \le C^{\kappa_p}M_\nu(2n), \quad n\ge n_p.
\end{equation}

Finally, let $x\ge n_p+1$ and put $n=\lfloor x\rfloor$. Then $n\ge n_p$,
$n\le x$, and, since $n\ge1$,
\[
 2x<2n+2\le4n.
\]
Using the monotonicity of $M_\nu$ and applying
\eqref{eq:integer-moment-doubling} twice, we get
\begin{align*}
 M_\nu(x) &\le M_\nu(n) \le C^{\kappa_p}M_\nu(2n) \le C^{2\kappa_p}M_\nu(4n) \le C^{2\kappa_p}M_\nu(2x).
\end{align*}
This proves \eqref{eq:real-moment-doubling} and completes the proof.
\end{proof}

The following lemma is standard in the literature (see, e.g., \cite{PR21}) and we include a proof for the reader's convenience.
\begin{lemma}\label{lem:moment-to-tail}
Let $\nu$ be a finite positive Borel measure on $[0,1)$ such that
$\nu([r,1))>0$ for every $0\le r<1$. If there exist constants $x_0\ge1$ and $C\ge1$ such that
\begin{equation*}
 M_\nu(x)\le C M_\nu(2x),
 \quad x\ge x_0.
\end{equation*}
Then there is an explicit threshold $r_0=r_0(x_0,C)\in[1/2,1)$ such that
\begin{equation}\label{eq:U-tail-doubling}
 \nu([r,1)) \le16C^4\nu\!\left(\left[\frac{1+r}{2},1\right)\right),
 \quad r_0\le r<1.
\end{equation}
\end{lemma}
\begin{proof}
Fix $1/2\le r<1$ and set
\[
s = \frac{1+r}{2}.
\]
Choose $$x=x(r):=\frac{\log(2C)}{-\log s},$$ so that $C s^x = 1/2$. Moreover, the condition $x \ge x_0$ is equivalent to $s\ge(2C)^{-1/x_0}$, or, equivalently,
\[
r\ge r_0:=\max\left\{\frac12,\,2(2C)^{-1/x_0}-1\right\}.
\]
Clearly $r_0<1$, and $x\ge x_0$ whenever $r_0\le r<1$.

Splitting $M_\nu(2x)$ at $s$, we obtain
\begin{align}
M_\nu(2x)
&=\int_{[0,s)}t^{2x} \, \dd \nu(t) +\int_{[s,1)}t^{2x}\, \dd \nu(t) \notag\\
&\le s^x\int_{[0,s)} t^x\, \dd \nu(t)+\nu([s,1)) \le s^x M_\nu(x) + \nu([s,1)).\label{eq:splitting-at-s}
\end{align}
If $r_0\le r<1$, then $x \ge x_0$. The doubling property of $M_\nu$ and \eqref{eq:splitting-at-s} gives
\[
M_\nu(x) \le C M_\nu(2x) \le Cs^x M_\nu(x) + C\nu([s,1)).
\]
Hence, for $r_0\le r<1$,
\begin{align*}
M_\nu(x) \le Cs^xM_\nu(x)+C\nu([s,1)) =\frac12M_\nu(x)+C\nu([s,1)),
\end{align*}
and therefore
\[
M_\nu(x)\le 2C\nu([s,1)).
\]
On the other hand,
\[
M_\nu(x)= \int_{0}^{1} t^x \, \dd \nu(t) \ge \int_{r}^1 t^x \, \dd \nu(t) \ge r^x\nu([r,1)),
\]
so
\[
\nu([r,1)) \le 2C r^{-x}\nu([s,1)).
\]
Since $r\ge1/2$, we have $$8r-(1+r)^3 =(1-r)(r^2+4r-1)\ge0,$$ and consequently
\[
r\ge\left(\frac{1+r}{2}\right)^3=s^3.
\]
It follows that $r^{-x}\le s^{-3x}=(2C)^3$. Therefore,
\[
\nu([r,1)) \le16C^4 \nu\!\left(\left[\frac{1+r}{2},1\right)\right), \quad r_0 \le r < 1.
\]
This proves \eqref{eq:U-tail-doubling} and completes the proof.
\end{proof}

\begin{proof}[Proof of $(3)\Rightarrow (4)$ in Theorem~\ref{thm:main-ZS}]
By duality, we may assume that $P_\omega$ is bounded on $L_\omega^p$ for some $p>2$. 
We first establish the reverse H\"{o}lder inequality \eqref{eq:omega-moment-reverse-holder}, which was obtained by Dostanić \cite{Dos09}. For completeness, we give a proof by testing functions.

For $n\in\mathbb N$, consider a family of functions 
\[
 f_n(re^{i\theta}) =r^{n/(p-1)}e^{in\theta}.
\]
Since $|f_n|\le1$, we have $$f_n\in L_\omega^p \cap L_\omega^2.$$ The functions
$
e_k(z)=\frac{z^k}{\sqrt{2\omega_{2k+1}}}$, $ k\ge0,
$
form an orthonormal basis of $A_\omega^2$. Hence,
\[
P_\omega f_n=\sum_{k=0}^{\infty}\langle f_n,e_k\rangle_{L_\omega^2}e_k.
\]
By angular orthogonality, $\langle f_n,e_k\rangle_{L_\omega^2}=0$ for $k\ne n$, while
\begin{align*}
\langle f_n,e_n\rangle_{L_\omega^2}
&=\frac{2}{\sqrt{2\omega_{2n+1}}}\int_0^1 r^{n+n/(p-1)+1}\omega(r)\,\dd r =\frac{2\omega_{p'n+1}}{\sqrt{2\omega_{2n+1}}}.
\end{align*}
Consequently,
\[
P_\omega f_n(z)=\frac{\omega_{p'n+1}}{\omega_{2n+1}}z^n.
\]
Moreover,
\begin{align*}
 \norm{f_n}_{L_\omega^p}^p
 &=2\int_0^1r^{np/(p-1)+1}\omega(r) \, \dd r =2 \omega_{p'n + 1},\\
 \norm{z^n}_{L_\omega^p}^p &=2\int_0^1r^{pn+1}\omega(r)\, \dd r =2\omega_{pn + 1}.
\end{align*}
Since $P_\omega$ is bounded on $L_\omega^p$, we have $$\|P_\omega f_n\|_{L_\omega^p} \le C \|f_n\|_{L_\omega^p}$$ for some constant $C = C(\omega, p) \ge1$. Therefore, 
\[
 \frac{\omega_{p'n + 1}}{\omega_{2n + 1}}\bigl(2\omega_{pn + 1}\bigr)^{1/p} \le C\bigl(2\omega_{p'n + 1}\bigr)^{1/p}.
\]
After cancelling $2^{1/p}\omega_{p'n+1}^{1/p}$, we obtain
\begin{equation}\label{eq:omega-moment-reverse-holder}
  \omega_{pn + 1}^{1/p}\, \omega_{p'n+1}^{1/p'} \le C\, \omega_{2n+1}.
\end{equation}
We now consider the measure
\[
 \dd \nu(r)=r\omega(r) \, \dd r.
\]
Since $\omega$ is integrable, $\nu$ is finite. Moreover, for $0< r<1$,
\[
\nu([r,1))=\int_r^1 t\omega(t)\,\dd t\ge r\widehat{\omega}(r)>0,
\]
and $\nu([0,1)) > 0$. Thus, $\nu$ satisfies the standing measure hypotheses of Lemmas~\ref{lem:reverse-Jensen} and~\ref{lem:moment-to-tail}. For this measure, write
\begin{equation}\label{eq:moment-of-rw}
 M(x):=M_\nu(x)=\int_0^1 r^x\,\dd\nu(r)=\omega_{x+1},
 \quad x\ge0.
\end{equation}
Hence, \eqref{eq:omega-moment-reverse-holder} is equivalent to the moment inequality
\[
 M(pn)^{1/p}M(p'n)^{1/p'}\le CM(2n).
\]
By Lemma~\ref{lem:reverse-Jensen} and Lemma~\ref{lem:moment-to-tail}, there are constants $C = C(\omega,p) \ge 1$ and $r_0 = r_0(\omega,p) \in [1/2,1)$ such that 
\[
\int_r^1 t \omega(t) \, \dd t \le C \int_{\frac{1 + r}{2}}^1 t \omega(t) \, \dd t, \quad \text{for all } r_0 \le r < 1.
\]
For $r_0 \le r < 1$,
\begin{align*}
  \widehat{\omega}(r) = \int_r^1 \omega(t) \,\dd t \le \frac{1}{r} \int_r^1 t \omega(t) \,\dd t \le \frac{C}{r} \int_{\frac{1 + r}{2}}^1 t\omega(t)\, \dd t \le \frac{C}{r} \int_{\frac{1 + r}{2}}^1 \omega(t) \,\dd t \le \frac{C}{r} \widehat{\omega}\left(\frac{1 + r}{2}\right).
\end{align*}
Since $r\ge r_0\ge \frac{1}{2}$,
\[
\widehat{\omega}(r) \le 2C \widehat{\omega}\left(\frac{1 + r}{2}\right), \quad \text{for all } r_0 \le r < 1.
\]
If $0 \le r < r_0$, then $\widehat{\omega}(r) \le \widehat{\omega}(0)$ and $\widehat{\omega} \left( \frac{1 + r}{2}\right) \ge \widehat{\omega} \left( \frac{1 + r_0}{2}\right) > 0$. Therefore, 
\[
\widehat{\omega}(r) \le \frac{\widehat{\omega}(0)}{\widehat{\omega} \left( \frac{1 + r_0}{2}\right)} \widehat{\omega} \left( \frac{1 + r}{2}\right).
\]
Finally, set
\[
C_{\omega, p} = \max\left\{2C,\frac{\widehat{\omega}(0)}{\widehat{\omega} \left( \frac{1 + r_0}{2}\right)} \right\}.
\]
Then, for all $0 \le r < 1$
\[
\widehat{\omega}(r) \le C_{\omega, p}\, \widehat{\omega}\left(\frac{1 + r}{2}\right).
\]
This completes the proof.
\end{proof}

\section{Weak-type (1,1) estimate for $\widehat{\mathcal{D}}$-Bergman projections} \label{S:S3}
The main goal of this section is to prove the following proposition, which 
establishes the implication $(4)\Rightarrow (1)$ in Theorem \ref{thm:main-ZS}
\begin{proposition}\label{pro:main-weak}
If $\omega\in\widehat{\mathcal{D}}$, then there exists a constant $C = C(\omega) \ge 1$ such that
\begin{equation*}
 \|P_\omega f\|_{L^{1,\infty}_\omega} \le C \|f\|_{L^1_\omega}, \quad f\in L^1_\omega.
\end{equation*}
\end{proposition}

The proof of Proposition \ref{pro:main-weak} employs the framework of Calderón–Zygmund operators on a space of non-homogeneous type. We recall these results below, following \cite{Hyt10}, \cite{HLYY12} and \cite{HM12}.

A triple $(X,d,\mu)$ is called a metric measure space if $(X,d)$ is a separable metric space and $\mu$ is a nonnegative
Borel measure on $X$. 
\begin{definition}
A metric space $(X,d)$ is called geometrically doubling if there exists a constant $N_0 \geq 1$ such that every ball $B(x,r)$ is covered by at most $N_0$ balls of radius $r/2$.
\end{definition}

\begin{remark}
Let \(X \subset \mathbb{R}^{n}\) be an open convex domain, equipped with the metric induced by the Euclidean metric on \(\mathbb{R}^{n}\). Then \(X\) is geometrically doubling. More precisely, every ball \(B_X(x,r)\) can be covered by at most \(5^{n}\) balls in \(X\) of radius \(r/2\). Thus, for $n = 2$, the unit disk $\mathbb{D}$, equipped with Euclidean metric $d(z,\zeta) = |z - \zeta|$, is geometrically doubling.
\end{remark}

\begin{definition}
A metric measure space $(X,d,\mu)$ is upper doubling with a dominating function $\lambda$ if there is a function
\[
 \lambda:X\times(0,\infty)\longrightarrow(0,\infty)
\]
and a constant $C = C(\lambda) \ge1$ such that, for every $x\in X$, the function
$s\mapsto\lambda(x,s)$ is nondecreasing and
\begin{equation*}
 \mu(B(x,s))
 \le\lambda(x,s)
 \le C \lambda(x,s/2),
 \quad s>0.
\end{equation*}
\end{definition}

\begin{remark}
As recalled in \cite[Remark~1.1(ii)]{HLYY12}, every upper-doubling
metric measure space admits a dominating function satisfying the
following local comparability property: there exists a constant
\(C = C(\lambda) \geq1\) such that
\begin{equation}\label{eq:general-local-comparability}
    \lambda(x,r) \leq C \lambda(y,r), \quad x,y\in X,\, d(x,y)\leq r.
\end{equation}
Throughout the paper, dominating functions are understood to satisfy
this local comparability property.
\end{remark}

\begin{definition}\label{def:sd}
Let $(X,d,\mu)$ be upper doubling with dominating function $\lambda$.  A
function
\[
 K:(X\times X)\setminus\{(x,x):x\in X\}\longrightarrow\C
\]
is a standard kernel of order $\tau\in(0,1]$ if there is a constant
$C = C(K,\lambda)$ such that
\begin{equation*}
 |K(x,y)|
 \le\frac{C}{\lambda(x,d(x,y))},
 \quad x\ne y,
\end{equation*}
and
\begin{align*}
|K(x,y)-K(x',y)| +|K(y,x)-K(y,x')| \le C \frac{d(x,x')^\tau}{d(x,y)^\tau\lambda(x,d(x,y))}
\end{align*}
whenever $d(x,y)\ge2d(x,x')$.
\end{definition}

Let \(K\) be a standard kernel. A linear operator \(T\), defined on
bounded measurable functions with bounded support, is called a
Calderón--Zygmund operator associated with \(K\) if
\[
Tf(x) = \int_X K(x,y)f(y)\,\dd \mu(y)
\]
for every bounded measurable function \(f\) with bounded support and every \(x\notin \operatorname{supp}f\).

The following result, proved by Hytönen, Liu, Yang, and Yang \cite{HLYY12}, plays a central role in our proof.  

\begin{lemma}[{\cite[Theorem~1.1]{HLYY12}}]\label{thm:HLYY}
Let $(X,d,\mu)$ be a separable, geometrically doubling, upper-doubling
metric measure space whose dominating function satisfies
\eqref{eq:general-local-comparability}. Assume that
\(\mu(\{x\})=0\) for every \(x\in X\). Let \(T\) be a
Calder\'on--Zygmund operator associated with a standard kernel. Then the following statements are equivalent:
\begin{itemize}
  \item $
    T:L^1_\mu\to L^{1,\infty}_\mu
$;
\item \(T\) is bounded on \(L^2_\mu\).
\end{itemize}
\end{lemma}

For a radial weight $\omega$, the induced positive measure on $\mathbb{D}$ is
\[
 \dd \mu_\omega(z) = \omega(z)\,\dd A(z).
\]

The disk $\mathbb{D}$ is separable and geometrically doubling, the measure
$\mu_\omega$ is non-atomic, and $P_\omega$ is bounded on $L_\omega^2$.
It therefore remains to verify the following two facts in order to apply
the theorem of Hyt\"onen, Liu, Yang, and Yang:

\begin{itemize}
  \item If \(\omega\in\widehat{\mathcal D}\), then
  \((\D,d,\mu_\omega)\) is upper doubling with a dominating function
  satisfying the local comparability property
  (Proposition~\ref{prop:upper-doubling}).
  \item The integral kernel of \(P_\omega\) is a standard kernel
  (Proposition~\ref{prop:standard-kernel}).
\end{itemize}

The first task begins with the following lemmas and culminates in Proposition \ref{prop:upper-doubling}.
Recall from \eqref{eq:widehat-D-definition} that, if
\(\omega\in\widehat{\mathcal D}\), then there exists a constant
\(C=C(\omega)\ge1\) such that
\begin{equation}\label{eq:tail-doubling-recalled}
  \widehat{\omega}(r)
  \le C\widehat{\omega}\left(\frac{1+r}{2}\right), \quad 0\le r<1.
\end{equation}
We shall use the following characterization of the $\widehat{\mathcal D}$ weight (see \cite[Lemma A(iii)]{PR16} and \cite[Lemma 2.1(vi)]{Pel16})
\begin{equation}\label{eq:standard-moment-tail}
  C^{-1} \widehat{\omega}\left(1-\frac1x\right)
  \le \omega_x
  \le C \widehat{\omega}\left(1-\frac1x\right), \quad x\ge1,
\end{equation}
for some constant $C$. Recall also $u_+:=\max\{u,0\}$ for  $u\in\mathbb R$.

\begin{lemma}\label{lem:tail-doubling-properties}
Let \(\omega\) be a radial weight in \(\widehat{\mathcal D}\). Then, for each fixed
\(z\in\overline{\mathbb D}\), the function $s\mapsto \widehat{\omega}\bigl((|z|-s)_+\bigr)$ is nondecreasing. Moreover, there exists a constant \(C=C(\omega)\ge1\), depending only on \(\omega\), such that
\begin{equation}\label{eq:tail-doubling}
  \widehat{\omega}\bigl((|z|-2t)_+\bigr)
  \le C\widehat{\omega}\bigl((|z|-t)_+\bigr), \quad z\in\overline{\mathbb D},\, t>0,
\end{equation}
and
\begin{equation}\label{eq:moment-tail-comparison}
  \widehat{\omega}\bigl((1-t)_+\bigr) \le 2C\omega_{1/t}, \quad 0<t<2.
\end{equation}
\end{lemma}

\begin{proof}
Fix \(z\in\overline{\mathbb D}\). If \(0<s_1\le s_2\), then $(|z|-s_1)_+ \ge (|z|-s_2)_+$.
Since \(\widehat{\omega}\) is nonincreasing, it follows that
\[
  \widehat{\omega}\bigl((|z|-s_1)_+\bigr)
  \le
  \widehat{\omega}\bigl((|z|-s_2)_+\bigr).
\]
Thus $s\mapsto \widehat{\omega}\bigl((|z|-s)_+\bigr)$ is nondecreasing.

We first prove \eqref{eq:tail-doubling}. Suppose that
\(0<t\le |z|/2\). Applying
\eqref{eq:tail-doubling-recalled} with \(r=|z|-2t\), we obtain
\[
  \widehat{\omega}\bigl((|z|-2t)_+\bigr) = \widehat{\omega}(|z|-2t) \le C\widehat{\omega} \left(\frac{1+|z|-2t}{2}\right) = C\widehat{\omega} \left(\frac{1+|z|}{2}-t\right).
\]
Since $ (1+|z|)/2 -t \ge |z|-t$ and \(\widehat{\omega}\) is nonincreasing, we conclude that
\[
  \widehat{\omega}\bigl((|z|-2t)_+\bigr) \le C\widehat{\omega}(|z|-t) = C\widehat{\omega}\bigl((|z|-t)_+\bigr).
\]

Suppose next that $\frac{|z|}{2}<t<|z|$. Then $(|z|-2t)_+=0$, and $0<|z|-t<\frac12$.
Applying \eqref{eq:tail-doubling-recalled} with \(r=0\), we obtain
\[
  \widehat{\omega}\bigl((|z|-2t)_+\bigr)= \widehat{\omega}(0) \le C\widehat{\omega}\left(\frac12\right) \le C \widehat{\omega}(|z|-t)  = C\widehat{\omega}\bigl((|z|-t)_+\bigr).
\]

Finally, if \(t\ge |z|\), then $(|z|-2t)_+ = (|z|-t)_+ = 0$, and \eqref{eq:tail-doubling} is immediate.

We next prove \eqref{eq:moment-tail-comparison}. Suppose first
that \(0<t\le1\). Setting \(x=1/t\) in the left-hand inequality
of \eqref{eq:standard-moment-tail}, we obtain $\widehat{\omega}(1-t) \le C\omega_{1/t}$.
Consequently,
\[
  \widehat{\omega}\bigl((1-t)_+\bigr) = \widehat{\omega}(1-t) \le C\omega_{1/t} \le 2C\omega_{1/t}.
\]

It remains to consider \(1<t<2\). In this case, $(1-t)_+=0$.
Applying \eqref{eq:tail-doubling-recalled} with \(r=0\), we have
\[
  \widehat{\omega}(0)
  \le
  C\widehat{\omega}\left(\frac12\right).
\]
Moreover, since \(1/t<1\), we have $s^{1/t}\ge s\ge\frac12$ for $\frac12\le s\le1$.
Therefore,
\[
  \omega_{1/t} \ge \int_{1/2}^1s^{1/t}\omega(s)\,\dd s \ge \frac12\int_{1/2}^1\omega(s)\,\dd s = \frac12 \widehat{\omega}\left(\frac12\right).
\]
It follows that
\[
  \widehat{\omega}\bigl((1-t)_+\bigr) =
  \widehat{\omega}(0)  \le C\widehat{\omega}\left(\frac12\right) \le 2C\omega_{1/t}.
\]
This proves \eqref{eq:moment-tail-comparison} and completes the
proof.
\end{proof}
The following elementary observations will be used repeatedly.

\begin{lemma}
  For every $z,\zeta\in\mathbb{D}$,
    \begin{equation}\label{eq:kernel-metric-distance}
      |1-\overline z\,\zeta| \ge |z-\zeta|.
    \end{equation}

  Consequently, if $z,\zeta, z' \in\mathbb{D}$ satisfy $|z-\zeta|\ge 2|z- z' |$,
    and if $\xi$ lies on the line segment joining $z$ and
    $ z' $, then
    \begin{equation}\label{eq:segment-kernel-distance}
      |1-\overline\zeta\,\xi|
      \ge \frac{|z-\zeta|}{2}.
    \end{equation}
\end{lemma}
\begin{comment}
\begin{proof}
  The identity
\[
 |1-\overline z\,\zeta|^{2}-|z-\zeta|^{2}
 =(1-|z|^{2})(1-|\zeta|^{2})
\]
immediately gives \eqref{eq:kernel-metric-distance}.

Suppose now that $|z-\zeta|\ge2|z- z' |$ and that $\xi$ lies on the line segment joining $z$ and
$ z' $. By \eqref{eq:kernel-metric-distance},
\begin{align*}
 |1-\overline\zeta\,\xi| \ge|\xi-\zeta| \ge|z-\zeta|-|\xi-z| \ge|z-\zeta|-|z- z' | \ge\frac{|z-\zeta|}{2}.
\end{align*}
This proves \eqref{eq:segment-kernel-distance}.
\end{proof}
\end{comment}
\begin{lemma}\label{lem:tail-size-estimates}
Let $\omega$ be a radial weight in $\widehat{\mathcal{D}}$, and let $ C \ge1$ be the constant in Lemma~\ref{lem:tail-doubling-properties}. Then, for every $z,\zeta\in\mathbb{D}$ with $z\ne\zeta$,
\begin{equation}\label{eq:tail-kernel-comparison}
 \widehat{\omega} \bigl((|z|-|z-\zeta|)_+\bigr) \le  C  \widehat{\omega} \bigl((1-|1-\overline z\,\zeta|)_+\bigr).
\end{equation}
Moreover, if $ z' \in\mathbb{D}$ satisfies $|z-\zeta|\ge2|z- z' |$ and $\xi$ lies on the line segment joining $z$ and $ z' $, then
\begin{equation}\label{eq:segment-tail-kernel-comparison}
 \widehat{\omega} \bigl((|z|-|z-\zeta|)_+\bigr) \le  C ^{2} \widehat{\omega} \bigl((1-|1-\overline\zeta\,\xi|)_+\bigr).
\end{equation}
\end{lemma}
\begin{proof}
Let \(z,\zeta\in\mathbb{D}\) with \(z\ne\zeta\). By
\eqref{eq:kernel-metric-distance}, $|z-\zeta|\le |1-\overline z\,\zeta|$.
Moreover,
\[
1-|z| \le 1-|z||\zeta| \le |1-\overline z\,\zeta|.
\]
Therefore,
\[
|z|-|z-\zeta| =1-\bigl(1-|z|+|z-\zeta|\bigr) \ge 1-2|1-\overline z\,\zeta|.
\]
Hence
\[
(|z|-|z-\zeta|)_+
\ge (1-2|1-\overline z\,\zeta|)_+.
\]
Since \(\widehat{\omega}\) is nonincreasing, it follows that
\[
\widehat{\omega}\bigl((|z|-|z-\zeta|)_+\bigr) \le
\widehat{\omega}\bigl((1-2|1-\overline z\,\zeta|)_+\bigr)\le  C \widehat{\omega}\bigl((1-|1-\overline z\,\zeta|)_+\bigr).
\]
This proves \eqref{eq:tail-kernel-comparison}.

Suppose next that \( z' \in\mathbb{D}\) satisfies $|z-\zeta|\ge2|z- z' |$
and that \(\xi\) lies on the line segment joining \(z\) and
\( z' \). Then
\[
|\xi-z|
\le |z- z' |
\le \frac{|z-\zeta|}{2}.
\]
By \eqref{eq:segment-kernel-distance}, $|1-\overline\zeta\,\xi| \ge |z-\zeta|/2$,
and hence
\[
|z-\zeta| \le2|1-\overline\zeta\,\xi| \quad\text{and}\quad |\xi-z| \le |1-\overline\zeta\,\xi|.
\]
Furthermore,
\begin{align*}
1-|z| &\le |1-\overline\zeta\,z| \le |1-\overline\zeta\,\xi| +|\overline\zeta(\xi-z)|\\
 &\le |1-\overline\zeta\,\xi|+|\xi-z| \le2|1-\overline\zeta\,\xi|.
\end{align*}
Consequently,
\[
|z|-|z-\zeta| =1-\bigl(1-|z|+|z-\zeta|\bigr) \ge1-4|1-\overline\zeta\,\xi|.
\]
Thus,
\[
(|z|-|z-\zeta|)_+
\ge(1-4|1-\overline\zeta\,\xi|)_+.
\]
Using again the fact that \(\widehat{\omega}\) is nonincreasing and applying
the doubling estimate twice, we obtain
\[
\widehat{\omega}\bigl((|z|-|z-\zeta|)_+\bigr)\le \widehat{\omega}\bigl((1-4|1-\overline\zeta\,\xi|)_+\bigr) \le  C ^{2} \widehat{\omega}\bigl((1-|1-\overline\zeta\,\xi|)_+\bigr).
\]
This proves \eqref{eq:segment-tail-kernel-comparison} and completes
the proof.
\end{proof}

\begin{proposition}\label{prop:upper-doubling}
Let \(\omega\) be a radial weight in \(\widehat{\mathcal D}\), and let
\(C\geq 1\) be the constant in
Lemma~\ref{lem:tail-doubling-properties}. For every \(z\in\D\) and
\(s>0\), define
\begin{equation*}
    \lambda_\omega(z,s) = 2s\,\widehat{\omega}\bigl((|z|-s)_+\bigr),
\end{equation*}
where \((u)_+=\max\{u,0\}\). Then the function \(s\mapsto\lambda_\omega(z,s)\) is nondecreasing, and
\begin{equation}\label{eq:measure-upper}
    \mu_\omega(B(z,s)\cap\D) \leq \lambda_\omega(z,s), \quad z\in\D,\, s>0.
\end{equation}
Moreover,
\begin{equation}\label{eq:lambda-upper-doubling}
    \lambda_\omega(z,s) \leq 2C\lambda_\omega(z,s/2), \quad z\in\D, \, s>0,
\end{equation}
and
\begin{equation}\label{eq:lambda-local-comparability}
    \lambda_\omega(z,s) \leq C\lambda_\omega(\zeta,s), \quad z,\zeta\in\D, \,|z-\zeta|\leq s.
\end{equation}
Consequently, \((\D,d,\mu_\omega)\) is upper doubling with dominating function \(\lambda_\omega\), which also satisfies the local comparability property.
\end{proposition}

\begin{proof}
For a fixed $z \in \mathbb{D}$, the function
$s \mapsto \widehat{\omega}((|z| - s)_+)$ is nondecreasing. Hence, the function $\lambda_{\omega}(z,s)$ with respect to the
variable $s$.

We next prove the inequality \eqref{eq:measure-upper}. Fix
\(z\in\mathbb D\), and write \(\rho=|z|\). Since both the measure
\(\mu_\omega\) and the Euclidean metric are invariant under
rotations, we may assume that \(z=\rho\). For \(0<r<1\), set
\[
 E_r = \left\{ \theta\in[-\pi,\pi): |re^{i\theta}-\rho|<s \right\}.
\]
We claim that
\[
 r|E_r|\le 2\pi s.
\]
Indeed, if \(r\le s\), then $r|E_r|\le 2\pi r\le 2\pi s$. Suppose that \(r>s\). If \(E_r=\varnothing\), the assertion is
immediate. If \(E_r\ne\varnothing\), then \(\rho>0\). For every
\(\theta\in E_r\),
\[
 r|\sin\theta| \le |re^{i\theta}-\rho| <s.
\]
Moreover, $r^2+\rho^2-2r\rho\cos\theta <s^2 <r^2$, so \(\cos\theta>0\). Hence
\[
 |E_r| \le 2\arcsin(s/r) \le \pi s/r,
\]
and therefore $ r|E_r| \le \pi s \le 2\pi s$.

It follows that
\begin{align*}
 \mu_{\omega}(B(z,s)\cap\D)
 &= \frac1\pi \int_{(\rho-s)_+}^{\min\{1,\rho+s\}} r\omega(r)|E_r|\,\dd r \\
 &\le 2s \int_{(\rho-s)_+}^{\min\{1,\rho+s\}} \omega(r)\,\dd r \le 2s \widehat{\omega} ((\rho-s)_+) = \lambda_\omega(z,s).
\end{align*}
This proves the inequality \eqref{eq:measure-upper}.

By Lemma \ref{lem:tail-doubling-properties}, we have
\begin{align*}
 \lambda_{\omega}(z,s) = 2s\widehat{\omega}((|z| - s)_+) \le 2 C s \widehat{\omega}((|z| - (s/2))_+) = 2 C \lambda_{\omega}(z,s/2).
\end{align*}
Hence, $(\mathbb D,d,\mu_\omega)$ is upper doubling.

Finally, let \(z,\zeta\in\D\) satisfy \(|z-\zeta|\leq s\). By the
reverse triangle inequality,
\[
  |z| \geq |\zeta|-|z-\zeta| \geq |\zeta|-s.
\]
Consequently, $(|z|-s)_+ \geq (|\zeta|-2s)_+$. Since \(\widehat{\omega}\) is nonincreasing, Lemma
\ref{lem:tail-doubling-properties} yields
\begin{align*}
  \widehat{\omega}\bigl((|z|-s)_+\bigr) \leq \widehat{\omega}\bigl((|\zeta|-2s)_+\bigr) \leq C\widehat{\omega}\bigl((|\zeta|-s)_+\bigr).
\end{align*}
Multiplying by \(2s\), we obtain
\[
  \lambda_\omega(z,s) \leq C\lambda_\omega(\zeta,s),
\]
which proves the inequality \eqref{eq:lambda-local-comparability}. The proof is complete now.
\end{proof}

Recall that
\[
P_\omega f(z) = \int_{\mathbb D} f(\zeta)\overline{B_z^\omega(\zeta)} \,\dd\mu_\omega(\zeta), \quad \dd\mu_\omega=\omega\,\dd A.
\]
Thus the integral kernel of \(P_\omega\) is
\[K_\omega(z,\zeta)=\overline{B_z^\omega(\zeta)}.\]
The second task is therefore to prove the following proposition.

\begin{proposition}\label{prop:standard-kernel}
There exists a constant $C>0$, depending only on
$\omega$, such that the integral kernel $(z,\zeta)\longmapsto\overline{B_z^{\omega}(\zeta)}$ is a standard kernel of order $1$ on $(\D,|\cdot|,\mu_{\omega})$ with dominating function $\lambda_\omega$. More precisely,
\begin{equation*}
 |B_z^{\omega}(\zeta)| \le \frac{C}{\lambda_\omega(z,|z-\zeta|)}, \quad z\ne\zeta,
\end{equation*}
and, whenever $|z-\zeta|\ge2|z- z'|$,
\begin{align*}
 |B_\zeta^{\omega}(z)-B_\zeta^{\omega}(z')| + |B_z^{\omega}(\zeta) - B_{z'}^{\omega}(\zeta)| \le C \frac{|z- z'|}{|z-\zeta|\lambda_\omega(z,|z-\zeta|)}.
\end{align*}
\end{proposition}

For the proof of Proposition \ref{prop:standard-kernel}, we need the weighted Bergman kernel estimate that is due to Pennanen, Rättyä, Wang, and Wu \cite{PRWW26}.

\begin{proposition}\label{lem:kernel-estimate}
  Let $\omega$ be a radial weight in $\widehat{\mathcal{D}}$.  Then there is a  constant $C = C(\omega) > 0$, depending only on $\omega$, such that, for every $z,\zeta\in\D$,
  \begin{equation}\label{eq:kernel-estimate}
    |B_z^{\omega}(\zeta)|
    \le \frac{C}{|1-\overline z\,\zeta|\,\omega_{1/|1-\overline z\,\zeta|}},
  \end{equation}
  and
  \begin{equation}\label{eq:kernel-estimate-derivative}
    |\partial_\zeta B_z^{\omega}(\zeta)|
    \le \frac{C}{|1-\overline z\,\zeta|^2\,\omega_{1/|1-\overline z\,\zeta|}}.
  \end{equation}  
\end{proposition}

Proposition \ref{lem:kernel-estimate} is central to establishing our weak-type (1,1) estimate. For the reader's convenience and completeness, we give an alternative proof in Section \ref{S:S4}.

\begin{proof}[Proof of Proposition~\ref{prop:standard-kernel}]
Let \(C_1\ge 1\) be the constant in Lemma~\ref{lem:tail-doubling-properties}, let \(C_2\ge 1\) be the
constant in Lemma~\ref{lem:tail-size-estimates}, and let \(C_3>0\) be the constant in Lemma~\ref{lem:kernel-estimate}. All these
constants depend only on \(\omega\).

Let \(z,\zeta\in\mathbb D\) with \(z\ne\zeta\). By \eqref{eq:kernel-metric-distance}, \eqref{eq:tail-kernel-comparison}, and
\eqref{eq:moment-tail-comparison}, we have
\[
\begin{aligned}
\lambda_\omega(z,|z-\zeta|)
&=2|z-\zeta|\, \widehat{\omega}\bigl((|z|-|z-\zeta|)_+\bigr)\le 2C_2|z-\zeta|\, \widehat{\omega}\bigl((1-|1-\overline z\,\zeta|)_+\bigr)\\
&\le 2C_2|1-\overline z\,\zeta|\, \widehat{\omega}\bigl((1-|1-\overline z\,\zeta|)_+\bigr)
\le4C_1C_2|1-\overline z\,\zeta|\,\omega_{1/|1-\overline z\,\zeta|}.
\end{aligned}
\]
Notice that \(0<|1-\overline z\,\zeta|<2\), so
\eqref{eq:moment-tail-comparison} is applicable. It follows from
\eqref{eq:kernel-estimate} that
\[
|B_z^\omega(\zeta)|
\le \frac{C_3}{|1-\overline z\,\zeta|\, \omega_{1/|1-\overline z\,\zeta|}}
\le\frac{4C_1C_2C_3}{\lambda_\omega(z,|z-\zeta|)}.
\]
Suppose now that $|z-\zeta|\ge 2|z-z'|$, and let \(\xi\) lie on the line segment joining \(z\) and \(z'\).
By \eqref{eq:segment-tail-kernel-comparison} and
\eqref{eq:moment-tail-comparison}, we have
\[
\begin{aligned}
\lambda_\omega(z,|z-\zeta|)
= 2|z-\zeta|\, \widehat{\omega} \bigl((|z|-|z-\zeta|)_+\bigr)
&\le 2C_2^2|z-\zeta|\, \widehat{\omega} \bigl((1-|1-\overline\zeta\,\xi|)_+\bigr) \\
&\le 4C_1C_2^2|z-\zeta|\, \omega_{1/|1-\overline\zeta\,\xi|}.
\end{aligned}
\]
Therefore, by \eqref{eq:kernel-estimate-derivative},
\[
|\partial_\xi B_\zeta^\omega(\xi)|
 \le\frac{C_3}{|1-\overline\zeta\,\xi|^2\, \omega_{1/|1-\overline\zeta\,\xi|}} \le \frac{4C_1C_2^2C_3|z-\zeta|}{|1-\overline\zeta\,\xi|^2\lambda_\omega(z,|z-\zeta|)}.
\]
By \eqref{eq:segment-kernel-distance}, $|1-\overline\zeta\,\xi| \ge |z-\zeta|/2$,
and hence $|z-\zeta|/ |1-\overline\zeta\,\xi|^2 \le 4/|z-\zeta|$. Consequently,
\[
|\partial_\xi B_\zeta^\omega(\xi)|
\le
\frac{16C_1C_2^2C_3}
{|z-\zeta|\lambda_\omega(z,|z-\zeta|)}.
\]
Integrating along the line segment joining \(z\) and \(z'\), we
obtain
\[
|B_\zeta^\omega(z)-B_\zeta^\omega(z')|
\le
16C_1C_2^2C_3
\frac{|z-z'|}
{|z-\zeta|\lambda_\omega(z,|z-\zeta|)}.
\]
Since the Bergman kernel is Hermitian, $B_z^\omega(\zeta) =\overline{B_\zeta^\omega(z)}$.
It follows that
\[
|B_\zeta^\omega(z)-B_\zeta^\omega(z')| + |B_z^\omega(\zeta)-B_{z'}^\omega(\zeta)|
\le 32C_1C_2^2C_3 \frac{|z-z'|}{|z-\zeta|\lambda_\omega(z,|z-\zeta|)}.
\]
Taking $C=\max\left\{4C_1C_2C_3,\,32C_1C_2^2C_3\right\}$, we conclude that \((z,\zeta)\mapsto\overline{B_z^\omega(\zeta)}\) is a standard kernel of order \(1\). The proof is complete now. 
\end{proof}

\begin{proof}[Proof of Proposition \ref{pro:main-weak}]
By Proposition~\ref{prop:upper-doubling},
$(\D,d,\mu_{\omega})$ is upper doubling with dominating function
$\lambda_\omega$, which satisfies the local comparability property. The metric space $\D$ is
separable and geometrically doubling. Since the measure $\mu_{\omega}$ is absolutely continuous with respect to Lebesgue measure, $\mu_{\omega}(\{z\})=0$, $z \in \mathbb{D}$.

By Proposition~\ref{prop:standard-kernel}, the kernel
$(z,\zeta)\mapsto\overline{B_z^{\omega}(\zeta)}$ is a standard kernel of
order $1$. Moreover, the Bergman projection defined by 
\[
 P_\omega f(z)
 =\int_\D f(\zeta)\overline{B_z^{\omega}(\zeta)}\,\dd\mu_{\omega}(\zeta)
\]
is bounded on $L^2_\omega$. Therefore Lemma \ref{thm:HLYY} applies. Thus, $P_{\omega} : L^1_\omega \to L^{1,\infty}_\omega$ is bounded. This proves Proposition \ref{pro:main-weak} and establishes $(4)\Rightarrow (1)$ in Theorem \ref{thm:main-ZS}.
\end{proof}

\section{The maximal weighted Bergman projection}\label{S:corollaryP+}
The goal of this section is to prove Corollary \ref{thm:P+main-ZS} through the following strategy
\[
(1)\Rightarrow (2) \Rightarrow (3) \Rightarrow (4) \Rightarrow (1).
\]
\begin{proof}
Suppose that (1) holds, that is the maximal radial weighted Bergman projection $P_\omega^+$ is weak-type (1,1). Since
\begin{equation}\label{eq:ptop+}
|P_\omega f(z)|\leq P_\omega^+(|f|)(z)\qquad \forall z\in\mathbb{D},
\end{equation}
we have that the weighted Bergman projection $P_\omega$ is weak-type (1,1). By Theorem \ref{thm:main-ZS}, $\omega$ is a $\widehat{\mathcal{D}}$-weight. Recall the integral kernel of $P_\omega^+$ is $|B_z^\omega(\zeta)|$. By Proposition \ref{prop:standard-kernel}, $|B_z^\omega(\zeta)|$ is a standard kernel. By Lemma \ref{thm:HLYY}, the weak-type (1,1) estimate of $P_\omega^+$ implies that it is $L^2_\omega$-bounded. By the Marcinkiewicz interpolation and duality, $P_\omega^+$ is $L^p_\omega$-bounded for all $1<p<+\infty$. 

The implication $(2)\Rightarrow (3)$ is trivial. 

Now assume that the condition (3) holds, i.e., the maximal radial Bergman projection $P_\omega^+$ is $L_\omega^p$-bounded for some $p\in(1,+\infty)\setminus\{2\}$. Then so is the Bergman projection $P_\omega$. By Theorem~\ref{thm:main-ZS}, we have $\omega\in\widehat{\mathcal{D}}$. Using the assumption (3) and the Pel\'{a}ez--R\"{a}tty\"{a} theorem, we obtain $\omega\in\mathcal{D}$ which completes $(3)\Rightarrow (4)$.

Assume that the condition $(4)$ holds, $\omega$ is a $\mathcal{D}$-weight. Since $\mathcal{D}\subset\widehat{\mathcal{D}}$, the Pel\'{a}ez-R\"{a}tty\"{a} theorem implies that $P_\omega^+$ is $L^2_\omega$-bounded. Then its weak-(1,1) estimate follows from Proposition \ref{prop:standard-kernel} and Lemma \ref{thm:HLYY}. 
The proof is complete now.
\end{proof}

\section{From weak-(1,1) to the $\widehat{\mathcal{D}}$ condition via testing functions}\label{S:imply14}
In Section \ref{S:S2}, the proof of $(1)\Rightarrow (4)$ in Theorem~\ref{thm:main-ZS} follows from the interpolation theorem and $L_\omega^p$-boundedness. One might ask for a more direct proof. Here, we provide such a proof by using testing functions. The argument may be of independent interest.

\begin{lemma}\label{lem:Laplace-doubling}
Let $F:[0,\infty)\to[0,\infty)$ be bounded and nondecreasing, with $F(0)=0$ and $F(t)>0$ for every $t>0$.  Suppose that $C\ge1$ and that
\begin{equation}\label{eq:Laplace-hypothesis}
 \sup_{s>0}e^{-ns}F(s)
 \le 2Cn\int_0^\infty e^{-2nt}F(t)\,\dd t, \quad \text{for every integer }n\ge1.
\end{equation}
Then, $F$ has the doubling property
\begin{equation}\label{eq:Laplace-doubling}
 F(t)\le64eC^3F(t/2),
 \quad 0<t\le1.
\end{equation}
\end{lemma}

\begin{proof}
For each integer $n\ge1$, the assumptions on $F$ imply that $0 < \sup_{s>0}e^{-ns}F(s) < \infty$.  Define
\[
 G_n(t)=\frac{e^{-nt}F(t)}{\sup_{s>0}e^{-ns}F(s)}, \quad t\ge0.
\]
Then $0\le G_n(t)\le1$.  Dividing
\eqref{eq:Laplace-hypothesis} by $\sup_{s>0}e^{-ns}F(s)$ and making the change of variables
$v=nt$ gives
\begin{align*}
 1
 &\le2Cn\int_0^\infty e^{-nt}G_n(t)\,\dd t =2C\int_0^\infty e^{-v}G_n(v/n)\,\dd v.
\end{align*}
Consequently,
\begin{equation}\label{eq:G-average}
 \int_0^\infty e^{-v}G_n(v/n)\,\dd v \ge\frac{1}{2C}.
\end{equation}
Put $L=\log(4C)$. Since $G_n\le1$, we have
\[
 \int_L^\infty e^{-v}G_n(v/n)\,\dd v \le e^{-L}=\frac{1}{4C}.
\]
It follows from \eqref{eq:G-average} and a weighted averaging argument that
\[
 \int_0^L e^{-v}G_n(v/n)\,\dd v \ge\frac{1}{4C}.
\]
Hence there is $v_n\in[0,L]$ such that $ G_n(v_n/n)\ge 1/4C$.
Indeed, if $G_n(v/n)<1/(4C)$ for all $v\in[0,L]$, then
\[
\int_0^L e^{-v}G_n(v/n)\,\dd v \le \frac{1}{4C} \int_0^L e^{-v}\,\dd v < \frac{1}{4C},
\]
contradicting the previous inequality. By definition of $G_n$, writing $u_n=v_n/n$, we obtain
\begin{equation}\label{eq:An-localized}
 \sup_{s>0}e^{-ns}F(s) \le4Ce^{-nu_n}F(u_n), \quad 0\le u_n\le\frac{L}{n}.
\end{equation}

Fix $0<t\le1$ and choose $n=\left\lceil\frac{2L}{t}\right\rceil$.
Then $nt\le2L+1$, $ 0\le u_n\le\frac{L}{n}\le\frac{t}{2}$.
By \eqref{eq:An-localized}, we have 
\begin{align*}
 e^{-nt}F(t) &\le \sup_{s>0}e^{-ns}F(s) \le4Ce^{-nu_n}F(u_n).
\end{align*}
Since $F$ is nondecreasing,
\begin{align*}
 F(t)\le4Ce^{n(t-u_n)}F(u_n) \le4Ce^{nt}F(t/2) \le4Ce^{2L+1}F(t/2) =64eC^3F(t/2).
\end{align*}
This proves \eqref{eq:Laplace-doubling}, and completes the proof.
\end{proof}

\begin{proof}[Proof of $(1)\Rightarrow (4)$ in Theorem~\ref{thm:main-ZS}]
Assume that there is a constant $C>0$ such that
\begin{equation}\label{eq:weak-converse-assumption}
 \| P_{\omega} f\|_{L^{1,\infty}_{\omega}}
 \le C\|f\|_{L^1_{\omega}},
 \quad f\in L^1_{\omega}.
\end{equation}
Since $ P_{\omega}1=1$, applying \eqref{eq:weak-converse-assumption} to the
constant function gives $C\ge1$. Indeed, \(\|1\|_{L^{1,\infty}_\omega}=\mu_\omega(\mathbb D)=\|1\|_{L^1_\omega}\).

Fix an integer $n\ge1$ and $0<\rho<1$, and consider the test function 
\[
 f_{\rho,n}(re^{i\theta}) =\mathbbm{1}_{[\rho,1)}(r)e^{in\theta}.
\]
Then
\begin{equation}\label{eq:test-L1-norm}
 \|f_{\rho,n}\|_{L^1_{\omega}} =2\int_\rho^1r\omega(r)\,\dd r.
\end{equation}
The series representation of $B_z^{\omega}$ and angular orthogonality give
\begin{align*}
 P_\omega f_{\rho,n}(z)
 &=\int_\rho^1\int_0^{2\pi} e^{in\theta} \sum_{k=0}^\infty \frac{z^kr^ke^{-ik\theta}}{2\omega_{2k+1}} \frac{\dd\theta}{\pi}\,r\omega(r)\,\dd r\notag\\
 &=\frac{\displaystyle\int_\rho^1r^{n+1}\omega(r)\,\dd r} {\omega_{2n+1}}z^n.
\end{align*}
Moreover, 
\begin{align}
 \|z^n\|_{L^{1,\infty}_{\omega}}
 &=\sup_{0<\alpha<1}\alpha\,
 \mu_{\omega}\bigl(\{z\in\D:|z|^n>\alpha\}\bigr)\notag\\
 &=\sup_{0<\alpha<1}2\alpha
 \int_{\alpha^{1/n}}^1r\omega(r)\,\dd r = \sup_{0<s<1}2 s^n \int_{s}^1 r\omega(r) \,\dd r .
 \label{eq:monomial-weak-norm}
\end{align}

Applying \eqref{eq:weak-converse-assumption} to $f_{\rho,n}$ and using
\eqref{eq:test-L1-norm}--\eqref{eq:monomial-weak-norm}, we obtain
\begin{equation}\label{eq:test-weak-estimate}
 \frac{\displaystyle\int_\rho^1r^{n+1}\omega(r)\,\dd r}{\omega_{2n+1}}\sup_{0<s<1}2 s^n \int_{s}^1 r\omega(r) \,\dd r
 \le 2C \int_\rho^1r\omega(r)\,\dd r.
\end{equation}
The numerator satisfies
\[
 \int_\rho^1r^{n+1}\omega(r)\,\dd r \ge\rho^n\int_\rho^1r\omega(r)\,\dd r > 0.
\]
Dividing by $\int_{\rho}^{1} r \omega(r) \, \dd r$  and using \eqref{eq:test-weak-estimate} yields $\rho^n \sup_{0<s<1}2 s^n \int_{s}^1 r\omega(r) \,\dd r \le2C\omega_{2n+1}$, and
letting $\rho$ tend to $1$ gives
\begin{equation}\label{eq:tail-moment-necessary}
 \sup_{0<s<1}2 s^n \int_{s}^1 r\omega(r) \,\dd r \le2C\omega_{2n+1}, \quad n\ge1.
\end{equation}

Define the function 
\[
F(t) = 2\int_{e^{-t}}^1 r \omega(r) \, \dd r, \quad t \ge 0.
\]
The function $F$ is bounded and nondecreasing, $F(0)=0$, and $F(t)>0$
for $t>0$.  In addition,
\begin{equation}\label{eq:An-Laplace-form}
 \sup_{t>0}e^{-nt}F(t) = \sup_{t > 0}2 e^{-nt} \int_{e^{-t}}^1 r \omega(r) \, \dd r = \sup_{0<s<1}2 s^n \int_{s}^1 r\omega(r) \,\dd r.
\end{equation}
Tonelli's theorem gives the exact identity
\begin{align}
 2n\int_0^\infty e^{-2nt}F(t)\,\dd t
 =4n\int_0^\infty e^{-2nt} \int_{e^{-t}}^1r\omega(r)\,\dd r\,\dd t 
 &=4n\int_0^1r\omega(r)\int_{-\log r}^\infty e^{-2nt}\,\dd t\,\dd r\notag\\
 &=2\int_0^1r^{2n+1}\omega(r)\,\dd r=2\omega_{2n+1}.
 \label{eq:moment-Laplace-identity}
\end{align}
Combining \eqref{eq:tail-moment-necessary},
\eqref{eq:An-Laplace-form}, and \eqref{eq:moment-Laplace-identity}, we find
\[
 \sup_{t>0}e^{-nt}F(t)
 \le2Cn
 \int_0^\infty e^{-2nt}F(t)\,\dd t,
 \quad n\ge1.
\]
By Lemma~\ref{lem:Laplace-doubling} there is a constant $C_1 = 64eC^3$ such that 
\begin{equation}\label{eq:F-doubling-converse}
 F(t)\le C_1 F(t/2),
 \quad 0<t\le1.
\end{equation}
Let \(e^{-1}\le r<1\), and put \(t=-\log r\). Applying
\eqref{eq:F-doubling-converse} first with \(t\) and then with \(t/2\),
we obtain
\begin{align*}
 2 \int_{r}^1 s \omega(s) \, \dd s \le 2C_1 \int_{\sqrt{r}}^1 s \omega(s) \, \dd s  \le 2 C_1^2 \int_{r^{1/4}}^1 s \omega(s)\, \dd s.
\end{align*}
Moreover, for \(e^{-1}\le r<1\),
\[
 r^{1/4}\ge \frac{1+r}{2}.
\]
Indeed, writing \(x=r^{1/4}\), this is equivalent to
\[
 2x-1-x^4=(1-x)(x^3+x^2+x-1)\ge0.
\]
The first factor is nonnegative, while
$x\mapsto x^3+x^2+x-1$ is strictly increasing on $(0,\infty)$ and is
positive at $x=e^{-1/4}$. Hence the inequality holds for
$e^{-1/4}\le x\le1$. It follows that
\[
\int_{r}^1 s \omega(s) \, \dd s \le C_1^2 \int_{\frac{1 + r}{2}}^1 s \omega(s) \, \dd s.
\]
Since $e^{-1}\le r<1$, we have $r^{-1} \le e$. Therefore, the definition of $\widehat{\omega}$ gives
\begin{align}
  \widehat{\omega} (r) \le e\int_r^1 s \omega(s)\,\dd s \le  eC_1^2 \int_{\frac{1 + r}{2}}^1 s \omega(s) \, \dd s \le eC_1^2 \widehat{\omega}\left(\frac{1 + r}{2}\right).
 \label{eq:tail-doubling-near-boundary}
\end{align}

It remains to consider \(0\le r<e^{-1}\). In this range, $ (1+r)/2 \le (1+e^{-1})/2$.
Since \( \widehat{\omega} \) is nonincreasing,
\[
  \widehat{\omega} \left(\frac{1+r}{2}\right)
 \ge
  \widehat{\omega} \left(\frac{1+e^{-1}}{2}\right),
\]
and hence
\begin{equation}\label{eq:tail-doubling-interior}
  \widehat{\omega} (r) \le  \widehat{\omega} (0) \le \frac{ \widehat{\omega} (0)} { \widehat{\omega} ((1+e^{-1})/2)}  \widehat{\omega} \left(\frac{1+r}{2}\right).
\end{equation}
The denominator is positive by the standing assumption $\widehat{\omega}(r) > 0$ for $0 \le r < 1$. Combining \eqref{eq:tail-doubling-near-boundary} and \eqref{eq:tail-doubling-interior} and taking $C_{\omega} = \max\left\{eC_1^2,\, \frac{ \widehat{\omega} (0)}{ \widehat{\omega} ((1+e^{-1})/2)}\right\}$, we obtain
\[
  \widehat{\omega} (r)
 \le
  C_{\omega}  \widehat{\omega} \left(\frac{1+r}{2}\right),
 \quad 0\le r<1.
\]
Thus \(\omega\in\widehat{\mathcal D}\). This completes the proof.
\end{proof}

\section{Pointwise estimates for $\widehat{\mathcal D}$-weighted Bergman kernels} \label{S:S4}
Let $\omega\in\widehat{\mathcal D}$. By the moment characterization
of $\widehat{\mathcal D}$ \cite[Lemma 2.1]{Pel16}, there exists a constant
$C=C(\omega)\geq1$ such that
\begin{equation}\label{eq:moment-doubling}
    \omega_t\leq C\omega_{2t},\quad t \ge 0.
\end{equation}
 
For the measure $\dd\nu(r)=r\omega(r)\,\dd r$, the moment
function defined in \eqref{eq:moment-of-rw} is:
\[
 M(x)=\int_0^1r^x\,\dd\nu(r)=\omega_{x+1}, \quad x\ge0.
\]
Consider the smooth coefficient function
\begin{equation}\label{eq:scf}
 \Phi(x)=\frac{1}{2M(2x)}, \quad x\ge0.
\end{equation}
Then the reproducing kernel has the representation
\begin{equation}\label{eq:B-series}
 B_z^{\omega}(\zeta)
 =\sum_{n=0}^{\infty}\frac{(\overline z\,\zeta)^n}{2\omega_{2n+1}}
 =\sum_{n=0}^{\infty}\Phi(n)(\overline z\,\zeta)^n,
 \quad z,\zeta\in\D.
\end{equation}

The purpose of this section is to give a self-contained proof of Proposition~\ref{lem:kernel-estimate}. For convenience, we restate the estimate below.

\begin{proposition}\label{thm:apdx-kernel-estimate}
  Let $\omega$ be a radial weight in $\widehat{\mathcal{D}}$. Then there is a  constant $C = C(\omega) > 0$, depending only on $\omega$, such that, for every $z,\zeta\in\D$,
  \begin{equation}\label{eq:apdx-kernel-estimate}
    |B_z^{\omega}(\zeta)|
    \le \frac{C}{|1-\overline z\,\zeta|\,\omega_{1/|1-\overline z\,\zeta|}},
  \end{equation}
  and
  \begin{equation}\label{eq:apdx-kernel-estimate-derivative}
    |\partial_\zeta B_z^{\omega}(\zeta)|
    \le \frac{C}{|1-\overline z\,\zeta|^2\,\omega_{1/|1-\overline z\,\zeta|}}.
  \end{equation}
\end{proposition}

The proof is based on the series representation
\eqref{eq:B-series}, together with a dyadic decomposition and discrete
summation by parts. Since $M$ is decreasing on $[0,\infty)$, the
coefficient function
$
 \Phi(x)
$
is increasing. To carry out the summation-by-parts argument, we also
need scale-invariant estimates for the derivatives of $\Phi$. These
will be deduced from corresponding estimates for the moment function
$M$, which we establish first.
\begin{lemma}\label{lem:M-derivatives}
For every integer $k\ge1$ and every $x \ge 2$, there is a constant $C = C(\omega) > 0$ such that
\begin{equation*}
 |M^{(k)}(x)|
 \le  C \left(\frac{2k}{e}\right)^k
 \frac{M(x)}{x^k}.
\end{equation*}
\end{lemma}

\begin{proof}
Differentiation under the integral sign gives
\[
 M^{(k)}(x)=\int_0^1(\log r)^kr^x\,\dd\nu(r)
 =\int_0^1(\log r)^kr^{x+1}\omega(r)\,\dd r.
\]
To justify the differentiation, fix a compact interval
$[a,b]\subset(0,\infty)$. For $x\in[a,b]$ and $0<r<1$,
\[
 |\log r|^kr^x
 \le |\log r|^kr^a
 \le \left(\frac{2k}{ea}\right)^kr^{a/2}.
\]
The last function is bounded on $[0,1)$ and hence integrable with respect
to the finite measure $\nu$. The Dominated Convergence Theorem therefore applies on
every compact subinterval of $(0,\infty)$.

Writing $u=-\log r$, we have $$ \sup_{u\ge0}u^ke^{-xu/2}=\left(\frac{2k}{ex}\right)^k .$$
Consequently,
\[
 |\log r|^kr^x
 \le \left(\frac{2k}{ex}\right)^kr^{x/2},
\]
and hence
\[
 |M^{(k)}(x)|
 \le \left(\frac{2k}{ex}\right)^kM(x/2).
\]
Since $x\ge2$, the moment-doubling estimate \eqref{eq:moment-doubling}
and the monotonicity of $M$ give
\[
 M(x/2)=\omega_{x/2+1}
 \le  C \omega_{x+2}
 = C M(x+1)
 \le  C M(x).
\]
This completes the proof.
\end{proof}

\begin{lemma}\label{lem:Phi}
The function $\Phi$ defined by \eqref{eq:scf} satisfies the doubling property: there is a constant $C = C(\omega) > 0$ such that 
\begin{equation}\label{eq:Phi-doubling}
 \Phi(2x)\le  C \Phi(x), \quad x \ge 0.
\end{equation}
Moreover, for each integer $k\ge0$ there is a constant $C_k = C(\omega,k)>0$,
depending only on $ \omega $ and $k$, such that
\begin{equation}\label{eq:Phi-derivatives}
 |\Phi^{(k)}(x)| \le C_k \frac{\Phi(x)}{x^k}, \quad x\ge1.
\end{equation}
\end{lemma}

\begin{proof}
By \eqref{eq:moment-doubling} and the monotonicity of $M$,
\[
 M(2x)=\omega_{2x+1}
 \le  C \omega_{4x+2}
 = C M(4x+1)
 \le  C M(4x).
\]
Taking reciprocals and using $\Phi(x)=1/(2M(2x))$ gives
\[
 \Phi(2x)=\frac{1}{2M(4x)}
 \le \frac{ C }{2M(2x)}
 = C \Phi(x),
\]
which proves \eqref{eq:Phi-doubling}.

For the derivative estimate, set $g(x)=M(2x)$. Lemma
\ref{lem:M-derivatives} gives, for $j\ge1$ and $x\ge1$,
\begin{align*}
 |g^{(j)}(x)|
 &=2^j|M^{(j)}(2x)| \\
 &\le 2^j C \left(\frac{2j}{e}\right)^j
 \frac{M(2x)}{(2x)^j}
 = C \left(\frac{2j}{e}\right)^j
 \frac{g(x)}{x^j}.
\end{align*}
Because $\Phi(x)g(x)=1/2$, differentiating this identity $k$ times yields
\begin{equation}\label{eq:Phi-induction-identity}
 \Phi^{(k)}(x)g(x)
 =-\sum_{j=1}^k\binom{k}{j}
 \Phi^{(k-j)}(x)g^{(j)}(x).
\end{equation}
The case $k=0$ holds with $C_0=1$. Suppose that the estimate has
already been proved up to order $k-1$. Substituting those estimates and
the preceding bound for $g^{(j)}$ into
\eqref{eq:Phi-induction-identity} gives
\[
 |\Phi^{(k)}(x)|
 \le\left[ \sum_{j=1}^k\binom{k}{j} C_{k-j} C \left(\frac{2j}{e}\right)^j \right]
 \frac{\Phi(x)}{x^k}.
\]
Thus the induction closes after taking the expression in brackets as
$C_k$. In particular, each $C_k$ is produced by the argument and depends
only on $\omega$ and $k$. This completes the proof.
\end{proof}

We use a single smooth cutoff throughout the argument. Define
\[
 \rho(t)
 =\begin{cases}
  e^{-1/t},&t>0,\\
  0,&t\le0,
 \end{cases}
 \quad \eta(x)=\frac{\rho(2-x)}{\rho(2-x)+\rho(x-1)}.
\]
Then $\eta\in C^\infty(\mathbb R)$, $0\le\eta\le1$, $\eta$ is
nonincreasing, $\eta(x)=1$ for $x\le1$, and $\eta(x)=0$ for $x\ge2$.
Every derivative of $\eta$ is bounded.

For a real number $N\in [1,\infty)$, define
\begin{equation*}
 \eta_0^{N}(x)=\eta(x/N),
 \quad 
 \eta_j^{N}(x) =\eta\left(\frac{x}{2^jN}\right) -\eta\left(\frac{x}{2^{j-1}N}\right),
 \quad j\ge1.
\end{equation*}

\begin{lemma}\label{lem:dyadic-partition}
For every $N\ge1$ and $x\ge0$,
\begin{equation*}
 \sum_{j=0}^\infty\eta_j^{N}(x)=1.
\end{equation*}
Moreover, $\operatorname{supp} \eta_0^{N}\cap[0,\infty)\subset[0,2N]$, and, for $j\ge1$, $\eta_j^N$ has compact support, that is,
\begin{equation*}
 \operatorname{supp} \eta_j^{N}\subset[2^{j-1}N,2^{j+1}N].
\end{equation*}
For every integer $q\ge0$ and $j\ge1$,
\begin{equation}\label{eq:dyadic-derivatives}
 |(\eta_j^{N})^{(q)}(x)|
 \le(1+2^q)\|\eta^{(q)}\|_{L^\infty(\mathbb R)}(2^jN)^{-q},
 \quad x\in\mathbb R.
\end{equation}
\end{lemma}
\begin{proof}
We first prove the partition-of-unity identity.  By definition, for every integer \(J\ge1\),
\[
 \eta_0^N(x)+\sum_{j=1}^J\eta_j^N(x) 
 = \eta\left(\frac{x}{N}\right) +\sum_{j=1}^J \left[\eta\left(\frac{x}{2^jN}\right)-\eta\left(\frac{x}{2^{j-1}N}\right)\right] 
 = \eta\left(\frac{x}{2^JN}\right).
\]
For $x\ge0$ there exists an integer \(J_0\) such that \(2^{J_0}N \ge x\). Then for all \(J \ge J_0\), we have \(x/(2^JN) \le 1\), and thus \(\eta(x/(2^JN)) = 1\). Therefore, taking the limit as \(J \to \infty\) gives
\[
\sum_{j=0}^\infty\eta_j^{N}(x)=1, \quad x \ge 0.
\]
By the definition of $\eta$,
\[
 \operatorname{supp} \eta_0^N\cap[0,\infty)\subset[0,2N]
\quad \text{and} \quad \operatorname{supp} \eta_j^N\subset[2^{j-1}N,2^{j+1}N], \text{ for } j\ge1.
\]

Finally, let \(q\ge0\) be an integer.  By the chain rule,
\[
 \frac{\dd^q}{\dd x^q}\eta\left(\frac{x}{2^jN}\right)
 =(2^jN)^{-q}\eta^{(q)}\left(\frac{x}{2^jN}\right),
\]
and
\[
 \frac{\dd^q}{\dd x^q}\eta\left(\frac{x}{2^{j-1}N}\right)
 =(2^{j-1}N)^{-q}
  \eta^{(q)}\left(\frac{x}{2^{j-1}N}\right).
\]
Therefore,
\[
 (\eta_j^N)^{(q)}(x) = (2^jN)^{-q}\eta^{(q)}\left(\frac{x}{2^jN}\right) - (2^{j-1}N)^{-q} \eta^{(q)}\left(\frac{x}{2^{j-1}N}\right).
\]
Taking absolute values and using
\((2^{j-1}N)^{-q}=2^q(2^jN)^{-q}\), we obtain
\begin{align*}
 \bigl|(\eta_j^N)^{(q)}(x)\bigr|
 &\le (2^jN)^{-q} \left|\eta^{(q)}\left(\frac{x}{2^jN}\right)\right|+2^q(2^jN)^{-q}\left|\eta^{(q)}\left(\frac{x}{2^{j-1}N}\right)\right| \\
 &\le(1+2^q)\|\eta^{(q)}\|_{L^\infty(\mathbb R)}(2^jN)^{-q}.
\end{align*}
The proof is complete.
\end{proof}

We shall repeatedly use the following elementary form of discrete summation
by parts.

\begin{lemma}\label{lem:finite-differences}
Let $m\ge1$, let $F\in C^m(\mathbb R)$ have compact support, and set
$c_n=F(n)$ for $n\in\mathbb Z$.  Assume that $c_n=0$ for $n<0$.  With
$\nabla c_n=c_n-c_{n-1}$, one has, for every $w\in\D$,
\begin{equation}\label{eq:finite-summation-parts}
 (1-w)^m\sum_{n=0}^\infty c_nw^n =\sum_{n=0}^\infty\nabla^mc_nw^n.
\end{equation}
Furthermore,
\begin{equation}\label{eq:finite-difference-integral}
 \nabla^mc_n =\int_{[0,1]^m} F^{(m)}(n-s_1-\cdots-s_m) \,\dd s_1\cdots \dd s_m.
\end{equation}
If $\operatorname{supp} F\subset[a,b]$, then
\begin{equation}\label{eq:finite-difference-l1}
 \sum_{n=0}^\infty|\nabla^mc_n|
 \le (b-a+m+2)\|F^{(m)}\|_{L^\infty(\mathbb R)}.
\end{equation}
\end{lemma}
\begin{proof}
Since \(F\) has compact support, the sequence \((c_n)_{n\in\mathbb Z}\)
has finite support. Consequently, for every \(k\ge 0\), the sequence
\((\nabla^k c_n)_{n\in\mathbb Z}\) also has finite support, so all the
series below are in fact finite sums and index shifts are legitimate.

Using \(c_n=0\) for \(n<0\), and in particular \(c_{-1}=0\), we obtain
\[
\sum_{n=0}^{\infty}\nabla c_n w^n =\sum_{n=0}^{\infty}(c_n-c_{n-1})w^n =(1-w)\sum_{n=0}^{\infty}c_nw^n.
\]
Moreover, \(\nabla^k c_n=0\) for \(n<0\) for every \(k\ge 0\); this follows inductively from the definition of $\nabla$ and the assumption $c_n=0$ for all $n<0$. 
Applying the preceding identity successively to the sequences
\((c_n)\), \((\nabla c_n)\), \ldots, \((\nabla^{m-1}c_n)\), we find
\[
\sum_{n=0}^{\infty}\nabla^m c_nw^n
=(1-w)^m\sum_{n=0}^{\infty}c_nw^n,
\]
which proves \eqref{eq:finite-summation-parts}.

We next prove \eqref{eq:finite-difference-integral}. By the fundamental
theorem of calculus,
\[
\nabla c_n
=F(n)-F(n-1)
=\int_{n-1}^{n}F'(t)\,\dd t
=\int_0^1 F'(n-s)\,\dd s.
\]
More generally, we claim that for every \(1\le k\le m\),
\[
\nabla^k c_n
=\int_{[0,1]^k}
F^{(k)}(n-s_1-\cdots-s_k)
\,\dd s_1\cdots \dd s_k.
\]
The case \(k=1\) was just established. Suppose that the formula holds
for some \(k<m\). Then, using Fubini's theorem and the fundamental
theorem of calculus,
\[
\begin{aligned}
\nabla^{k+1}c_n
&=\nabla^k c_n-\nabla^k c_{n-1}\\
&=\int_{[0,1]^k}\Bigl[F^{(k)}(n-s_1-\cdots-s_k)-F^{(k)}(n-1-s_1-\cdots-s_k)\Bigr]\,\dd s_1\cdots \dd s_k\\
&=\int_{[0,1]^k}\int_0^1 F^{(k+1)}(n-s_1-\cdots-s_k-s_{k+1})\,\dd s_{k+1}\,\dd s_1\cdots \dd s_k\\
&=\int_{[0,1]^{k+1}}F^{(k+1)}(n-s_1-\cdots-s_{k+1})\,\dd s_1\cdots \dd s_{k+1}.
\end{aligned}
\]
Thus the claim follows by induction, and taking \(k=m\) gives
\eqref{eq:finite-difference-integral}.

Finally, assume that \(\operatorname{supp} F\subset[a,b]\). Since $F \in C^m(\mathbb{R})$,
\(\|F^{(m)}\|_{L^\infty(\mathbb R)} <  \infty\). Moreover, because \(F\) vanishes identically outside \([a,b]\), so does \(F^{(m)}\). Hence, if
\(\nabla^m c_n\neq 0\), then for some \((s_1,\ldots,s_m)\in[0,1]^m\) one has
\[
n-s_1-\cdots-s_m\in[a,b].
\]
Because \(0\le s_1+\cdots+s_m\le m\), this implies $a\le n\le b+m$.
On the other hand, \eqref{eq:finite-difference-integral} yields
\[
|\nabla^m c_n|
\le
\int_{[0,1]^m}
\bigl|F^{(m)}(n-s_1-\cdots-s_m)\bigr|
\,\dd s_1\cdots \dd s_m
\le \|F^{(m)}\|_{L^\infty(\mathbb R)}.
\]
The interval \([a,b+m]\) contains at most \(b-a+m+2\) integers.
Therefore,
\[
\sum_{n=0}^{\infty}|\nabla^m c_n|
\le (b-a+m+2)\|F^{(m)}\|_{L^\infty(\mathbb R)},
\]
which is \eqref{eq:finite-difference-l1}. This completes the proof.
\end{proof}
\begin{lemma}\label{lem:direct-kernel}
There is a constant $C = C(\omega) >0$, depending only on $\omega$, such that, for every $z,\zeta\in\D$,
\begin{equation}\label{eq:direct-kernel}
 \left|B_z^{\omega}(\zeta)\right|
 \le C
 \frac{2}{|1-\overline z\,\zeta|}
 \Phi\left(\frac{2}{|1-\overline z\,\zeta|}\right),
\end{equation}
and
\begin{equation}\label{eq:direct-kernel-derivative}
 \left|\partial_\zeta B_z^{\omega}(\zeta)\right|
 \le C
 \left(\frac{2}{|1-\overline z\,\zeta|}\right)^2
 \Phi\left(\frac{2}{|1-\overline z\,\zeta|}\right).
\end{equation}
\end{lemma}

\begin{proof}
Fix $z,\zeta\in\D$ and set
\[
 N=\frac{2}{|1-\overline z\,\zeta|}.
\]
Since $|\overline z\,\zeta|<1$, we have $0<|1-\overline z\,\zeta|<2$, and hence $N>1$.

Repeated use of \eqref{eq:Phi-doubling} shows that $\Phi(n)$ has at most
polynomial growth. Therefore,
\[
 \sum_{n=0}^{\infty}\Phi(n)|\overline z\,\zeta|^n<\infty.
\]
By \eqref{eq:B-series} and Lemma~\ref{lem:dyadic-partition}, we have $\sum_{j=0}^{\infty}\eta_j^N(n) = 1$ for every $n \ge 0$. Hence, 
\[
\sum_{j = 0}^{\infty} \sum_{n = 0}^{\infty} \Phi(n) \eta_j^N(n) |\overline{z}\zeta|^n = \sum_{n=0}^{\infty}\Phi(n)|\overline z\,\zeta|^n<\infty.
\]
Thus, by Fubini's theorem
\begin{equation*}
 B_z^{\omega}(\zeta)
 =\sum_{j=0}^{\infty}B_j(z,\zeta),
 \quad
 B_j(z,\zeta)
 =\sum_{n=0}^{\infty}
 \Phi(n)\eta_j^N(n)(\overline z\,\zeta)^n.
\end{equation*}

We first estimate the two low-frequency blocks. Since
$\operatorname{supp} \eta_0^N\cap[0,\infty)\subset[0,2N]$, the interval $[0,2N]$
contains at most $2N+1$ nonnegative integers. By \eqref{eq:Phi-doubling} in Lemma \ref{lem:Phi}, there is a constant $C = C(\omega) \ge 1$ such that, 
\begin{align*}
 |B_0(z,\zeta)|
 &\le \sum_{\substack{n\in\mathbb Z_{\ge0}\\ n\le2N}}\Phi(n) \le(2N+1)\Phi(2N)\le3 C N\Phi(N).
\end{align*}
Thus, we have the estimates
\begin{equation}\label{eq:low-block-zero}
 |B_0(z,\zeta)|
 \le3 C 
 \frac{2}{|1-\overline z\,\zeta|}
 \Phi\left(\frac{2}{|1-\overline z\,\zeta|}\right).
\end{equation}
Similarly, for $j = 1$, we have $\operatorname{supp} \eta_1^N\subset[N,4N]$. The interval $[N,4N]$
contains at most $3N+1\le4N$ integers, and
\[
 \Phi(n)\le\Phi(4N)\le  C ^2\Phi(N),
 \quad N\le n\le4N.
\]
Therefore,
\begin{equation}\label{eq:low-block-one}
 |B_1(z,\zeta)|
 \le4 C ^2N\Phi(N)
 =4 C ^2
 \frac{2}{|1-\overline z\,\zeta|}
 \Phi\left(\frac{2}{|1-\overline z\,\zeta|}\right).
\end{equation}

We now fix $j\ge2$ and choose
\[
 N_0=\left\lceil\log_2(4 C )\right\rceil,
 \quad
 \frac{2 C }{2^{N_0}}\le\frac12.
\]
Define
\[
 F_j(x)=
 \begin{cases}
  \Phi(x)\eta_j^N(x),&x>0,\\
  0,&x\le0.
 \end{cases}
\]
Since $\eta_j^N$ vanishes identically on
$(-\infty,2^{j-1}N]$, this zero extension belongs to
$C_c^\infty(\mathbb R)$ and
\[
 \operatorname{supp} F_j\subset[2^{j-1}N,2^{j+1}N].
\]
For $x\in[2^{j-1}N,2^{j+1}N]$ and $0\le q\le N_0$,
\eqref{eq:Phi-derivatives}, the monotonicity of $\Phi$, and
\eqref{eq:Phi-doubling} give
\begin{align*}
 |\Phi^{(q)}(x)| \le C_q\frac{\Phi(x)}{x^q} \le  C 2^qC_q \Phi(2^jN)(2^jN)^{-q}.
\end{align*}
Together with \eqref{eq:dyadic-derivatives} and Leibniz' formula,
\begin{align*}
|F_j^{(N_0)}(x)|
&= \left| \sum_{q=0}^{N_0}\binom{N_0}{q} \Phi^{(q)}(x) (\eta_j^N)^{(N_0-q)}(x) \right| \\
&\le \sum_{q=0}^{N_0} \binom{N_0}{q} |\Phi^{(q)}(x)| \bigl|(\eta_j^N)^{(N_0-q)}(x)\bigr| \\
&\le  C \sum_{q=0}^{N_0}\binom{N_0}{q}2^qC_q\Phi(2^jN)(2^jN)^{-q} \\
&\hspace{3cm}\times(2^jN)^{-(N_0-q)}\bigl(1+2^{N_0-q}\bigr)\|\eta^{(N_0-q)}\|_{L^\infty(\mathbb R)} \\
&= C \left(\sum_{q=0}^{N_0}\binom{N_0}{q}2^qC_q\bigl(1+2^{N_0-q}\bigr)\|\eta^{(N_0-q)}\|_{L^\infty(\mathbb R)}\right)\Phi(2^jN)(2^jN)^{-N_0}.
\end{align*}
This implies
\begin{equation}\label{eq:derivative-of-Fj}
 \|F_j^{(N_0)}\|_{L^\infty(\mathbb R)} \le \widetilde{C}_0 \Phi(2^jN)(2^jN)^{-N_0},
\end{equation}
where
\[
 \widetilde{C}_0 = C \sum_{q=0}^{N_0}\binom{N_0}{q} 2^qC_q\bigl(1+2^{N_0-q}\bigr) \|\eta^{(N_0-q)}\|_{L^\infty(\mathbb R)}.
\]
Recall $\operatorname{supp} F_j \subset [2^{j-1}N, 2^{j + 1}N]$ and the length of the interval is $3 \cdot 2^{j-1} N$. Thus, by Lemma~\ref{lem:finite-differences}, we have 
\begin{align*}
 \sum_{k=0}^{\infty}|\nabla^{N_0}F_j(k)|\le (3 \cdot 2^{j-1}N + N_0 + 2)\|F_j^{(N_0)}\|_{L^\infty(\mathbb R)}
\end{align*}
and \eqref{eq:derivative-of-Fj} yields
\begin{align*} \sum_{k=0}^{\infty}|\nabla^{N_0}F_j(k)| &\le (3\cdot2^{j-1}N+N_0+2)\widetilde C_0 \Phi(2^jN)(2^jN)^{-N_0}\\
&\le (N_0+4)\widetilde C_0 \Phi(2^jN)(2^jN)^{1-N_0}.
\end{align*}
Applying \eqref{eq:finite-summation-parts} with
$\overline z\,\zeta$ in place of its scalar variable gives
\begin{equation*}
 |B_j(z,\zeta)|
 \le(N_0+4)\widetilde{C}_0(2^jN)\Phi(2^jN)
 \bigl(|1-\overline z\,\zeta|\,2^jN\bigr)^{-N_0}.
\end{equation*}
Repeated use of \eqref{eq:Phi-doubling} gives
$\Phi(2^jN)\le  C ^j\Phi(N)$, while the definition of $N$ gives
\[
 |1-\overline z\,\zeta|\,2^jN=2^{j+1}.
\]
Consequently,
\begin{align*}
 |B_j(z,\zeta)|
 &\le(N_0+4)\widetilde{C}_0N\Phi(N)2^{-N_0}
 \left(\frac{2 C }{2^{N_0}}\right)^j\\
 &\le(N_0+4)\widetilde{C}_0\,2^{-j}N\Phi(N),
 \quad j\ge2.
\end{align*}
Summing over $j\ge2$ and combining the result with
\eqref{eq:low-block-zero} and \eqref{eq:low-block-one} proves
\eqref{eq:direct-kernel}. 

We next estimate the $\zeta$-derivative. The power series
\eqref{eq:B-series} and its termwise derivative converge locally
uniformly, because $\Phi(n)$ has at most polynomial growth. Hence
\[
 \partial_\zeta B_z^{\omega}(\zeta)
 =\overline z\sum_{n=1}^{\infty}
 n\Phi(n)(\overline z\,\zeta)^{n-1}.
\]
At the fixed point $(z,\zeta)$, we insert the dyadic partition into this
already differentiated series and write
\[
 \partial_\zeta B_z^{\omega}(\zeta)
 =\sum_{j=0}^{\infty}B_j^{[1]}(z,\zeta),
\]
where
\[
 B_j^{[1]}(z,\zeta)
 =\overline z\sum_{n=1}^{\infty}
 n\Phi(n)\eta_j^N(n)(\overline z\,\zeta)^{n-1}.
\]
This decomposition is justified by absolute convergence. 
The order of operations is important: we first differentiate the original series term by term with respect to $\zeta$, and only then insert the dyadic partition at the scale
$N=2/|1-\overline z\,\zeta|$.
In particular, we do not differentiate the functions $B_j(z,\zeta)$ defined above, since the scale $N$ itself depends on $\zeta$.

Using $|z|\le1$, the support of $\eta_0^N$, the monotonicity of $\Phi$,
and \eqref{eq:Phi-doubling} in Lemma \ref{lem:Phi}, we pick a constant $C \ge 1$ such that 
\begin{align*}
 |B_0^{[1]}(z,\zeta)| \le\sum_{\substack{n\in\mathbb Z_{\ge1}\\ n\le2N}} n\Phi(n)\le  C \Phi(N) \sum_{\substack{n\in\mathbb Z_{\ge1}\\ n\le2N}}n
 \le4 C N^2\Phi(N).
\end{align*}
Therefore,
\begin{equation}\label{eq:low-derivative-zero}
 |B_0^{[1]}(z,\zeta)|
 \le4 C 
 \left(\frac{2}{|1-\overline z\,\zeta|}\right)^2
 \Phi\left(\frac{2}{|1-\overline z\,\zeta|}\right).
\end{equation}
Similarly, on the support of $\eta_1^N$ we have
$n\le4N$ and $\Phi(n)\le  C ^2\Phi(N)$. Since $[N,4N]$
contains at most $4N$ integers,
\begin{equation}\label{eq:low-derivative-one}
 |B_1^{[1]}(z,\zeta)|
 \le16 C ^2N^2\Phi(N)
 =16 C ^2
 \left(\frac{2}{|1-\overline z\,\zeta|}\right)^2
 \Phi\left(\frac{2}{|1-\overline z\,\zeta|}\right).
\end{equation}

Set $G(x)=x\Phi(x)$. For
$x\in[2^{j-1}N,2^{j+1}N]$, we have
\[
 |G(x)|\le2 C \Phi(2^jN)(2^jN).
\]
For every integer $q\ge1$,
\[
 G^{(q)}(x)=x\Phi^{(q)}(x)+q\Phi^{(q-1)}(x),
\]
and therefore
\begin{equation}\label{eq:G-derivatives}
 |G^{(q)}(x)|
 \le  C 2^{q-1}\bigl(4C_q+qC_{q-1}\bigr)
 \Phi(2^jN)(2^jN)^{1-q}.
\end{equation}
Choose
\[
 N_1=\left\lceil\log_2(8 C )\right\rceil,
 \quad
 \frac{4 C }{2^{N_1}}\le\frac12.
\]
For $j\ge2$, define
\[
 \widetilde{F}_j(x)=
 \begin{cases}
  G(x+1)\eta_j^N(x+1),&x>-1,\\
  0,&x\le-1.
 \end{cases}
\]
Then $\widetilde{F}_j\in C_c^\infty(\mathbb R)$,
\[
 \operatorname{supp} \widetilde{F}_j
 \subset[2^{j-1}N-1,2^{j+1}N-1],
\]
and
\[
 \widetilde{F}_j(k)=(k+1)\Phi(k+1)\eta_j^N(k+1),
 \quad k\ge0.
\]
Leibniz's formula, \eqref{eq:dyadic-derivatives}, and
\eqref{eq:G-derivatives} give
\begin{equation}\label{eq:derivative-of-tilde-Fj} 
 \|\widetilde{F}_j^{(N_1)}\|_{L^\infty(\mathbb R)}
 \le \widetilde{C}_1\Phi(2^jN)(2^jN)^{1-N_1},
\end{equation}
where
\begin{align*}
 \widetilde{C}_1
 &:=2 C \bigl(1+2^{N_1}\bigr)\|\eta^{(N_1)}\|_{L^\infty(\mathbb R)}\\
 &\quad+ C \sum_{q=1}^{N_1}\binom{N_1}{q} 2^{q-1}\bigl(4C_q+qC_{q-1}\bigr) \bigl(1+2^{N_1-q}\bigr) \|\eta^{(N_1-q)}\|_{L^\infty(\mathbb R)}.
\end{align*}
Recall $\operatorname{supp} \widetilde{F}_j \subset [2^{j-1}N - 1, 2^{j + 1}N - 1]$ and the length of the interval is $3 \cdot 2^{j-1} N$. Thus, by Lemma~\ref{lem:finite-differences}, we have 
\begin{align*}
 \sum_{k=0}^{\infty}|\nabla^{N_1}\widetilde{F}_j(k)|\le (3 \cdot 2^{j-1}N + N_1 + 2)\|\widetilde{F}_j^{(N_1)}\|_{L^\infty(\mathbb R)}
\end{align*}
and \eqref{eq:derivative-of-tilde-Fj} yields
\begin{align*}
 \sum_{k=0}^{\infty}|\nabla^{N_1}\widetilde{F}_j(k)|
 &\le (3 \cdot 2^{j-1}N + N_1 + 2)\widetilde{C}_1\Phi(2^jN)(2^jN)^{1-N_1}\\
 &\le(N_1+4)\widetilde{C}_1 \Phi(2^jN)(2^jN)^{2-N_1}.
\end{align*}
After the index shift $k=n-1$,
\[
 B_j^{[1]}(z,\zeta)
 =\overline z\sum_{k=0}^{\infty}
 \widetilde{F}_j(k)(\overline z\,\zeta)^k.
\]
Thus \eqref{eq:finite-summation-parts} and $|z|\le1$ imply
\[
 |B_j^{[1]}(z,\zeta)|
 \le(N_1+4)\widetilde{C}_1(2^jN)^2\Phi(2^jN)
 \bigl(|1-\overline z\,\zeta|\,2^jN\bigr)^{-N_1}.
\]
Using
$\Phi(2^jN)\le  C ^j\Phi(N)$ and
$|1-\overline z\,\zeta|\,2^jN=2^{j+1}$, we obtain
\begin{align*}
 |B_j^{[1]}(z,\zeta)|
 &\le(N_1+4)\widetilde{C}_1N^2\Phi(N)2^{-N_1}
 \left(\frac{4 C }{2^{N_1}}\right)^j\\
 &\le(N_1+4)\widetilde{C}_1\,2^{-j}N^2\Phi(N),
 \quad j\ge2.
\end{align*}
Summing over $j\ge2$ and combining the result with
\eqref{eq:low-derivative-zero}--\eqref{eq:low-derivative-one} proves
\eqref{eq:direct-kernel-derivative}. One may take
\[
 \widetilde{C} = \widetilde{C}(\omega)
 =\max\left\{3 C +4 C ^2+\frac{N_0+4}{2}\widetilde{C}_0,\,4 C +16 C ^2+\frac{N_1+4}{2}\widetilde{C}_1\right\}.
\]
Since $C$ and the constants $C_q$ depend only on $ \omega $,
the constants obtained above have the asserted dependence.
\end{proof}

\begin{proof}[Proof of Proposition~\ref{thm:apdx-kernel-estimate}]
For $z,\, \zeta \in \mathbb{D}$, if $0<|1-\overline z\,\zeta|\le2$, then the
monotonicity of the moments and three applications of
\eqref{eq:moment-doubling} give
\[
 \Phi\left(\frac{2}{|1-\overline z\,\zeta|}\right)
 =\frac{1}{2\omega_{4/|1-\overline z\,\zeta|+1}}
\le\frac{1}{2\omega_{8/|1-\overline z\,\zeta|}}
 \le\frac{ C ^3}{2\omega_{1/|1-\overline z\,\zeta|}}.
\]
Combining this estimate with Lemma~\ref{lem:direct-kernel} proves
\eqref{eq:apdx-kernel-estimate} and
\eqref{eq:apdx-kernel-estimate-derivative}. The proof is completed now.
\end{proof}


\begin{thebibliography}{99}

\bibitem[BB78]{BB78}
D. Békollé and A. Bonami,
Inégalités à poids pour le noyau de Bergman,
\emph{C. R. Acad. Sci. Paris Sér. A-B}
\textbf{286} (1978), no.~18, 775--778.

\bibitem[BLT21]{BLT21}
J. Bonet, W. Lusky, and J. Taskinen,
Unbounded Bergman projections on weighted spaces with respect to exponential weights,
\emph{Integral Equations Operator Theory}
\textbf{93} (2021), no.~6, Paper No.~61, 20 pp.

\bibitem[Bor04]{Bor04}
A. Borichev,
On the Bekollé--Bonami condition,
\emph{Math. Ann.}
\textbf{328} (2004), no.~3, 389--398.

\bibitem[CRW76]{CRW76}
R. Coifman, R. Rochberg, and G. Weiss,
Factorization theorems for Hardy spaces in several variables,
\emph{Ann. of Math. (2)}
\textbf{103} (1976), no.~3, 611--635.

\bibitem[CP15]{CP15}
O. Constantin and J. Peláez,
Boundedness of the Bergman projection on $L^p$-spaces with exponential weights,
\emph{Bull. Sci. Math.}
\textbf{139} (2015), no.~3, 245--268.

\bibitem[Da22]{Da22}
G. Dall'Ara,
Around $L^1$ (un)boundedness of Bergman and Szeg\H{o} projections,
\emph{J. Funct. Anal.}
\textbf{283} (2022), no.~5, Paper No.~109550, 27 pp.

\bibitem[DHZZ01]{DHZZ01}
Y. Deng, L. Huang, T. Zhao, and D. Zheng,
Bergman projection and Bergman spaces,
\emph{J. Operator Theory}
\textbf{46} (2001), no.~1, 3--24.



\bibitem[Dos04]{Dos04}
M. Dostani\'c,
Unboundedness of the Bergman projections on $L^p$ spaces with exponential weights,
\emph{Proc. Edinb. Math. Soc. (2)}
\textbf{47} (2004), no.~1, 111--117.

\bibitem[Dos09]{Dos09}
M. Dostani\'c,
Boundedness of the Bergman projections on $L^p$ spaces with radial weights,
\emph{Publ. Inst. Math. (Beograd) (N.S.)}
\textbf{86(100)} (2009), 5--20.

\bibitem[FR74]{FR74}
F. Forelli and W. Rudin,
Projection on spaces of holomorphic functions in balls,
\emph{Indiana Univ. Math. J.}
\textbf{24} (1974), 593--602.

\bibitem[GW24]{GW24}
A. Green and N. Wagner,
Weighted estimates for the Bergman projection on planar domains,
\emph{Trans. Amer. Math. Soc.}
\textbf{377} (2024), no.~11, 8023--8048.

\bibitem[GHX04]{GHX04}
K. Guo, J. Hu, and X. Xu,
Toeplitz algebras, subnormal tuples and rigidity on reproducing
$\mathbb{C}[z_1,\ldots,z_d]$-modules,
\emph{J. Funct. Anal.}
\textbf{210} (2004), no.~1, 214--247.

\bibitem[Hed02]{Hed02}
H. Hedenmalm,
The dual of a Bergman space on simply connected domains,
\emph{J. Anal. Math.}
\textbf{88} (2002), 311--335.

\bibitem[HJS02]{HJS02}
H. Hedenmalm, S. Jakobsson, and S. Shimorin,
A biharmonic maximum principle for hyperbolic surfaces,
\emph{J. Reine Angew. Math.}
\textbf{550} (2002), 25--75.

\bibitem[HKZ00]{HKZ00}
H. Hedenmalm, B. Korenblum, and K. Zhu,
\emph{Theory of Bergman Spaces},
Graduate Texts in Mathematics, vol.~199,
Springer-Verlag, New York, 2000.



\bibitem[Hyt10]{Hyt10}
T. Hyt\"onen,
A framework for non-homogeneous analysis on metric spaces, and the RBMO space of Tolsa,
\emph{Publ. Mat.}
\textbf{54} (2010), no.~2, 485--504.

\bibitem[HLYY12]{HLYY12}
T. Hyt\"onen, S. Liu, D. Yang, and D. Yang,
Boundedness of Calder\'on--Zygmund operators on non-homogeneous metric measure spaces,
\emph{Canad. J. Math.}
\textbf{64} (2012), no.~4, 892--923.

\bibitem[HM12]{HM12}
T. Hyt\"onen and H. Martikainen,
Non-homogeneous $Tb$ theorem and random dyadic cubes on metric measure spaces,
\emph{J. Geom. Anal.}
\textbf{22} (2012), no.~4, 1071--1107.

\bibitem[LR96]{LR96}
P. Lin and R. Rochberg,
Trace ideal criteria for Toeplitz and Hankel operators on weighted Bergman spaces
with exponential type weights,
\emph{Pacific J. Math.}
\textbf{173} (1996), no.~1, 127--146.

\bibitem[McN94]{McN94}
J.~D. McNeal,
The Bergman projection as a singular integral operator,
\emph{J. Geom. Anal.}
\textbf{4} (1994), no.~1, 91--103.



\bibitem[Pel16]{Pel16}
J. Peláez,
Small weighted Bergman spaces,
\emph{Proceedings of the Summer School in Complex and Harmonic Analysis,
and Related Topics}
\textbf{22} (2016), 29--98.

\bibitem[PR14]{PR14}
J. Peláez and J. Rättyä,
Weighted Bergman spaces induced by rapidly increasing weights,
\emph{Mem. Amer. Math. Soc.}
\textbf{227} (2014), no.~1066, vi+124 pp.

\bibitem[PR15]{PR15}
J. Peláez and J. Rättyä,
Embedding theorems for Bergman spaces via harmonic analysis,
\emph{Math. Ann.}
\textbf{362} (2015), no.~1--2, 205--239.

\bibitem[PR16]{PR16}
J. Peláez and J. Rättyä,
Two weight inequality for Bergman projection,
\emph{J. Math. Pures Appl.} (9)
\textbf{105} (2016), no.~1, 102--130.

\bibitem[PR21]{PR21}
J. Peláez and J. Rättyä,
Bergman projection induced by radial weight,
\emph{Adv. Math.}
\textbf{391} (2021), Paper No.~107950, 70 pp.

\bibitem[PRW19]{PRW19}
J. Peláez, J. Rättyä, and B.~D. Wick,
Bergman projection induced by kernel with integral representation,
\emph{J. Anal. Math.}
\textbf{138} (2019), no.~1, 325--360.

\bibitem[PRWW26]{PRWW26}
A. Pennanen, J. Rättyä, S. Wang, and F. Wu,
Optimal off-diagonal upper estimates for Bergman reproducing kernels,
\emph{arXiv:2607.19959} (2026).

\bibitem[Sh02]{Sh02}
S. Shimorin,
An integral formula for weighted Bergman reproducing kernels,
\emph{Complex Var. Theory Appl.}
\textbf{47} (2002), no.~12, 1015--1028.

\bibitem[ZJ64]{ZJ64}
V. Zaharjuta and V. Judovic,
The general form of a linear functional in $H^p$,
\emph{Uspekhi Mat. Nauk}
\textbf{19} (1964), no.~2, 139--142
(in Russian).

\bibitem[Ze13]{Ze13}
Y. Zeytuncu, $L^p$ regularity of weighted Bergman projections,
\emph{Trans. Amer. Math. Soc.}
\textbf{365} (2013), no.~6, 2959--2976.

\bibitem[Zhu07]{Zhu07}
K. Zhu,
\emph{Operator Theory in Function Spaces}, 2nd ed., Mathematical Surveys and Monographs, vol.~138,
American Mathematical Society, Providence, RI, 2007.

\end{thebibliography}
\end{document}